\documentclass[11pt]{amsart}

\usepackage{amsmath,amsthm}

\usepackage{amssymb}
\usepackage[latin1]{inputenc}

\newcommand{\C}{{\mathbb C}}       % Field of complex numbers
\newcommand{\R}{{\mathbb R}}       % Field of real numbers
\newcommand{\Z}{{\mathbb Z}}       % Ring of integer numbers
\newcommand{\DD}{{\mathcal D}}
\newcommand{\HH}{{\mathcal H}}
\newcommand{\LL}{{\mathcal L}}

\newcommand{\ZZ}{{\mathcal Z}}

\newcommand{\CC}{{\mathcal C}}

\newcommand{\diam}{{\rm diam}}
\newcommand{\dist}{{\rm dist}}

\newcommand{\rf}[1]{{(\ref{#1})}}
\newcommand{\supp}{{\rm supp}}

\newcommand{\vphi}{{\varphi}}
\newcommand{\ve}{{\varepsilon}}
\newcommand{\vv}{{\vspace{2mm}}}
\newcommand{\vvv}{{\vspace{3mm}}}
\newcommand{\wt}[1]{{\widetilde{#1}}}
\newcommand{\wh}[1]{{\widehat{#1}}}

\newcommand{\meas}{{\measuredangle}}

\newcommand{\noi}{\noindent}
\newcommand{\pv}{{\rm p.v.}}

\newtheorem{theorem}{Theorem}[section]
\newtheorem*{theorema*}{Theorem A}
\newtheorem*{theoremb*}{Theorem B}
\newtheorem*{theoremc*}{Theorem C}
\newtheorem*{theoremd*}{Theorem D}
\newtheorem{lemma}[theorem]{Lemma}
\newtheorem{mlemma}[theorem]{Main Lemma}
\newtheorem{coro}[theorem]{Corollary}

\newtheorem*{claim*}{Claim}
\theoremstyle{definition}

\theoremstyle{remark}
\newtheorem{remark}[theorem]{\bf Remark}

\numberwithin{equation}{section}

\begin{document}

\title{Principal values for Riesz transforms and rectifiability}

\author[XAVIER TOLSA]{Xavier Tolsa}

\address{Instituci\'o Catalana de Recerca i Estudis Avan\c{c}ats (ICREA) and Departament de Matem\`atiques, Universitat Aut\`onoma de
Bar\-ce\-lo\-na, 08193 Bellaterra (Barcelona), Catalunya}

\email{xtolsa@mat.uab.cat}

\thanks{Partially supported by grants MTM2007-62817 (Spain) and 2005-SGR-00774 (Gene\-ra\-litat
de Catalunya)}

%\subjclass{Primary 42B20; Secondary 42B25}

\date{July, 2007.}

\begin{abstract}
Let $E\subset \R^d$ with $\HH^n(E)<\infty$, where $\HH^n$ stands for the $n$-dimensional Hausdorff measure.
In this paper we prove that $E$ is $n$-rectifiable if and only if
the limit $$\lim_{\ve\to0}\int_{y\in E:|x-y|>\ve} \frac{x-y}{|x-y|^{n+1}}\,d\HH^n(y)$$
exists $\HH^n$-almost everywhere in $E$.
To prove this result we obtain precise estimates from above and from below for the
$L^2$ norm of the $n$-dimensional Riesz transforms on Lipschitz graphs.
\end{abstract}

\maketitle

% *************************************************************************
% *************************************************************************
% *************************************************************************
% *************************************************************************

\section{Introduction}

Given $x\in\R^d$, $x\neq0$, we consider the signed Riesz kernel
$K(x) = x/|x|^{n+1}$, for an integer such that $0<n\leq d$. Observe that $K$ is a vectorial kernel. The
$n$-dimensional Riesz transform of a finite Borel measure $\mu$ on $\R^d$ is defined by
$$R^n \mu(x) = \int K(x-y)\,d\mu(y), \qquad x\not\in\supp(\mu).$$
Notice that the integral above may fail to be absolutely convergent
for $x\in\supp(\mu)$. For this reason one considers
 the $\ve$-truncated $n$-dimensional Riesz transform, for $\ve>0$:
$$R^n_\ve \mu(x) = \int_{|x-y|>\ve} K(x-y)\,d\mu(y), \qquad x\in\R^d.$$
The principal values are denoted by $${\rm p.v.} R^n\mu(x)=
\lim_{\ve\to0}R_\ve^n\mu(x),$$ whenever the limit exists.

One says that a subset $E\subset\R^d$ is $n$-rectifiable if there
exists a countable family of $n$-dimensional $\CC^1$ submanifolds
$\{M_i\}_{i\geq1}$ such that
$$\HH^n\Bigl(E\setminus \bigcup_i M_i\Bigr)=0,$$
where $\HH^n$ stands for the $n$-dimensional Hausdorff measure.

In this paper we are interested in the relationship between
rectifiability and Riesz transforms. One of our main results is the
following.

\begin{theorem} \label{corovp}
Let $E\subset \R^d$ with $\HH^n(E)<\infty$. Then $E$ is
$n$-rectifiable if, and only if, the principal value ${\rm p.v.}
R^n(\HH^n_{|E})(x)$ exists for $\HH^n$-almost every $x\in E$.
\end{theorem}

In fact, the ``only if'' part of the theorem (rectifiability implies
existence of principal values) was well known (see \cite{MPr}, for
example). On the other hand, under the additional assumption that
\begin{equation} \label{eqdensi}
\liminf_{r\to 0}\frac{\HH^n(B(x,r)\cap E)}{r^n}>0\qquad
\mbox{$\HH^n$-a.e. $x\in E$,}
\end{equation}
Mattila and Preiss proved \cite{MPr} that if the principal value
$\pv R^n(\HH^n_{|E})(x)$ exists $\HH^n$-almost everywhere in
$E$, then $E$ is rectifiable. Getting rid of the hypothesis
\rf{eqdensi} was an open problem raised by authors in \cite{MPr}.

Let us also remark that in the particular case $n=1$, Theorem
\ref{corovp} was previously proved in \cite{Tolsa-pams} (and in
\cite{Mattila-adv} under the assumption \rf{eqdensi}) using the
relationship between the Cauchy transform and curvature of measures
(for more information on this curvature, see \cite{Melnikov} and \cite{MV}, for example).
In higher dimensions the curvature method does not
work (see \cite{Farag}) and new techniques are required.

We do not know if Theorem \ref{corovp} holds if one replaces the
assumption on the existence of principal values for the Riesz
transforms by
$$\sup_{\ve>0} |R^n_\ve(\HH^n_{|E})(x)|<\infty \qquad\mbox{$\HH^n$-a.e. $x\in
E$.}$$ That this is the case for $n=1$ was shown in
\cite{Tolsa-pams} using curvature. However, for $n>1$ this is an
open problem that looks very difficult (probably, as difficult as
proving that the $L^2$ boundedness of Riesz transforms with respect
to $\HH^n_{|E}$ implies the $n$-rectifiability of $E$).

Given a Borel measure $\mu$ on $\R^d$, its upper and lower
$n$-dimensional densities are defined, respectively, by
$$\Theta_\mu^{n,*}(x)=\limsup_{r\to 0}\frac{\mu(B(x,r))}{r^n},\qquad
\Theta^{n}_{\mu,*}(x)=\liminf_{r\to 0}\frac{\mu(B(x,r))}{r^n}.$$ So
\rf{eqdensi} means that the lower $n$-dimensional densities with
respect to $\HH^n_{|E}$ is positive $\HH^n$-a.e. in $E$. We recall
that if $\HH^n(E)<\infty$, then
$$0<\Theta^{n,*}_{\HH^n_{|E}}(x)<\infty\qquad\mbox{$\HH^n$-a.e. $x\in
E$.}$$ However there are sets $E$ with $0<\HH^n(E)<\infty$ such that
the lower density $\Theta^{n}_{\HH^n_{|E},*}(x)$ vanishes for every
$x\in E$ (see \cite[Chapter 6]{Mattila-llibre}, for example).

Theorem \ref{corovp} is a particular case of the following somewhat
stronger result.

\begin{theorem}\label{teovp}
Let $\mu$ be a finite Borel measure on $\R^d$. Let $E\subset \R^d$
be such that for all $x\in E$ we have
$$0<\Theta_\mu^{n,*}(x)<\infty \quad\mbox{ and }\quad \exists \,{\rm p.v.}
R^n\mu(x).$$ Then $E$ is $n$-rectifiable.
\end{theorem}

Our arguments to prove Theorems \ref{corovp} and \ref{teovp} are
very different from the ones in \cite{MPr} and \cite{Mattila-adv},
which are based on the use of tangent measures. A fundamental step
in our proof consists in obtaining precise $L^2$ estimates of Riesz
transforms on Lipschitz graphs. In a sense, these $L^2$ estimates
play a role analogous to curvature of measures in \cite{Tolsa-pams}.
Loosely speaking, the second step of the proof consists of using
these $L^2$ estimates to construct a Lipschitz graph containing a
suitable piece of $E$, by arguments more or less similar to the ones
in \cite{Leger}.

To describe in detail the $L^2$ estimates mentioned above we need to
introduce some additional terminology. We denote the projection
 $$(x_1,\ldots,x_n,\ldots,x_d)\mapsto
(x_1,\ldots,x_n,0,\cdots,0)$$ by $\Pi$, and we set $\Pi^\bot=I-\Pi$.
We also denote
$$R^{n,\bot}\mu(x) = \Pi^\bot(R^n\mu(x))\quad\mbox{ and }\quad R^{n,\bot}_\ve\mu(x) = \Pi^\bot(R^n_\ve\mu(x)).$$
That is to say, $R^{n,\bot}\mu(x)$ and $R^{n,\bot}_\ve\mu(x)$ are made
up of the components of $R^n\mu(x)$ and $R^n_\ve\mu(x)$ orthogonal
to $\R^n$, respectively (we are identifying $\R^n$ with
$\R^n\times\{(0,\ldots,0)\}$).

\begin{theorem} \label{teolip}
Consider the $n$-dimensional Lipschitz graph $\Gamma:=\{(x,y)\in
\R^n \times \R^{d-n}:\,y=A(x)\}$, and let $d\mu(z)=
g(z)\,d\HH^n_{|\Gamma}(z)$, where $g(\cdot)$ is a function such that
$C_1^{-1}\leq g(z)\leq C_1$ for all $z\in\Gamma$. Suppose that $A$
has compact support. If $\|g-1\|_2\leq C_2\|\nabla A\|_2$ and
$\|\nabla A\|_\infty\leq \ve_0$,  with $0<\ve_0<1$ small enough
(depending on $C_2$), then we have
$$\|\pv R^{n,\bot}\mu\|_{L^2(\mu)}\approx \|\pv R^n\mu\|_{L^2(\mu)}\approx \|\nabla A\|_2.$$
\end{theorem}

Let us remark that the existence of the principal values $\pv R^n\mu$ $\mu$-a.e. under the assumptions of the
theorem is a well know fact.
If we take $g(x)\equiv1$, we obtain:

\begin{coro}\label{corofac}
Consider the $n$-dimensional Lipschitz graph $\Gamma:=\{(x,y)\in
\R^n \times \R^{d-n}:\,y=A(x)\}$, and let $\mu=\HH^n_{|\Gamma}$.
Suppose that $A$ has compact support. If $\|\nabla A\|_\infty\leq
\ve_0$, with $0<\ve_0\leq1$ small enough, then
$$\|\pv R^{n,\bot}\mu\|_{L^2(\mu)}\approx \|\pv R^n\mu\|_{L^2(\mu)}\approx\|\nabla A\|_2.$$
\end{coro}

The upper estimate $\|\pv R^n\mu\|_{L^2(\mu)}\lesssim \|\nabla
A\|_2$ is an easy consequence of some of the results from
\cite{Dorronsoro} and \cite{Tolsa-preprint} and also holds replacing
$\ve_0$ by any big constant (see Lemma \ref{lemupp} in Section
\ref{sec3} for more details). The lower estimate $\|\pv
R^{n,\bot}\mu\|_{L^2(\mu)}\gtrsim \|\nabla A\|_2$ is more difficult.
To prove it we use a Fourier type estimate as well as the
quasiorthogonality techniques developed in \cite{Tolsa-preprint}. In
particular, the coefficients $\alpha(Q)$ (see Section
\ref{secprelim} for the definition) introduced in that paper are an
important tool for the proof.

We remark that we do not know if the inequalities $\|\pv R^{n}\mu\|_{L^2(\mu)}\geq C_3^{-1} \|\nabla A\|_2$
or $\|\pv R^{n,\bot}\mu\|_{L^2(\mu)}\geq C_3^{-1} \|\nabla A\|_2$ in Theorem \ref{teolip} or Corollary \ref{corofac}
hold assuming $\|\nabla A\|_\infty\leq C_4$ instead  of $\|\nabla A\|_\infty\leq\ve_0$, with $C_4$ arbitrarily
large and $C_3$ possibly depending on $C_4$.

Obtaining lower estimates for the $L^2$ norm of $n$-dimensional Riesz transforms in $\R^d$ is
also important for other problems, such as the characterization of removable singularities for bounded
 analytic  functions (for $n=1$) and Lipschitz harmonic functions (for $n\geq1$).
For instance, in \cite{MT}, in order to characterize some Cantor sets which are removable for Lipschitz harmonic functions in $\R^{n+1}$ first one needs to get a lower estimate of the norm $\|\pv R^n \mu\|_{L^2(\mu)}$, where $\mu$ is the natural probability measure supported on the given Cantor set.
Analogous results for bilipschitz images of Cantor sets
are obtained in \cite{GPT}. See also \cite{ENV} for other recent results which involve lower
estimates of $L^2$ norms of Riesz transforms, and \cite{David-revista}, \cite{Tolsa-sem}, \cite{Volberg} for
other questions on removability of singularities of bounded analytic functions and Lipschitz harmonic functions.

The plan of the paper is the following. In Section \ref{secprelim} we introduce some preliminary notation and
state some results that will be needed in the rest of the paper. Sections \ref{sec3}-\ref{sec6}
are devoted to the proof of Theorem \ref{teolip}, while Theorem \ref{teovp} is proved in Sections
\ref{sec7}-\ref{sec10} by arguments inspired in part by the corona type constructions of \cite{Leger} and \cite{DS1}.

% *************************************************************************
% *************************************************************************
% *************************************************************************
% *************************************************************************

\section{Preliminaries}\label{secprelim}

As usual, in the paper the letter `$C$' stands for an absolute
constant which may change its value at different occurrences. On the
other hand, constants with subscripts, such as $C_1$, retain its
value at different occurrences. The notation $A\lesssim B$ means
that there is a positive absolute constant $C$ such that $A\leq CB$.
Also, $A\approx B$ is equivalent to $A\lesssim B \lesssim A$.

An open ball with center $x$ and radius $x$ is denoted by $B(x,r)$.
If we want to remark that this is an $n$-dimensional ball, we write
$B_n(x,r)$.

Given $f\in L^1_{loc}(\mu)$, we denote $R^n_\mu(f) = R^n(f\,d\mu)$
and
 $R_{\mu,\ve}^n(f) = R_\ve^n(f\,d\mu)$. Recall also the definition of the maximal Riesz transform:
$$R_*^n \mu(x) = \sup_{\ve>0} |R_\ve^n\mu(x)|.$$
To simplify notation, if $n$ is fixed, quite often we will also
write $R\mu(x)$ instead of $R^n\mu(x)$, and analogously with respect
to $R_\ve\mu$, $R^\bot\mu$, $R_*\mu$, $R_{\mu,\ve}(f)$, etc. We say that the Riesz transform operator $R_\mu$
is bounded in $L^2(\mu)$ if the truncated operators $R_{\mu,\ve}$ are bounded in $L^2(\mu)$ uniformly
on $\ve>0$.

Given $0<n\leq d$, we say that a Borel measure $\mu$ on $\R^d$ is
$n$-dimensional Ahlfors-David regular, or simply AD regular, if
there exists some constant $C_0$ such that
$C_0^{-1}r^n\leq\mu(B(x,r))\leq C_0r^n$ for all $x\in\supp(\mu)$,
$0<r\leq\diam(\supp(\mu))$. It is not difficult to see that such a
measure~$\mu$ must be of the form $d\mu=\rho\,d\HH^n_{|\supp(\mu)}$,
where $\rho$ is some positive function bounded from above and from
below.

Given $E\subset \C$ and a cube $Q\subset \R^d$,  we set
$$\beta_E(Q) = \inf_L \biggl\{ \sup_{y\in E\cap 3Q}
\frac{\dist(y,L)}{\ell(Q)}\biggr\},$$ where the infimum is taken
over all $n$-planes $L$ in $\R^d$. The $L^p$ version of $\beta$ is
the following,
$$\beta_{p,\mu}(Q) = \inf_L
\biggl\{ \frac1{\ell(Q)^n}\int_{3Q}
\biggl(\frac{\dist(y,L)}{\ell(Q)}\biggr)^pd\mu(y)\biggr\}^{1/p},$$
where the infimum is taken over all $n$-planes in $\R^d$ again.
In
our paper we will have $E=\supp(\mu)$ and, to simplify notation, we
will write $\beta$ (or $\beta_\infty$) and $\beta_p$ instead of
$\beta_E$ and $\beta_{p,\mu}$. The definition of $\beta_p(B)$ for a ball $B$ is analogous to the one
of $\beta_p(Q)$ for a cube $Q$.

\begin{remark} Consider
the $n$-dimensional Lipschitz graph $\Gamma:=\{(x,y)\in \R^n \times
\R^{d-n}:\,y=A(x)\}$, and let $d\mu(z)= d\HH^n_{|\Gamma}(z)$.
Suppose that $\|\nabla A\|_\infty\leq C_5$. By \cite[Theorem
6]{Dorronsoro}, we have
 $$\|\nabla A\|_2^2 \approx \sum_{Q\in\DD}\beta_1(Q)^2\mu(Q) \approx
 \sum_{Q\in\DD}\beta_2(Q)^2\mu(Q),$$
 with constants depending only on $C_5$.
 \end{remark}

Given a set $A\subset \R^d$ and two Borel measures $\sigma$, $\nu$
on $\R^d$ , we set
$$\dist_A(\sigma,\nu):= \sup\Bigl\{ \Bigl|{\textstyle \int f\,d\sigma  -
\int f\,d\nu}\Bigr|:\,{\rm Lip}(f) \leq1,\,\supp(f)\subset
A\Bigr\}.$$

Given a Borel measure $\mu$ on $\R^d$ and a cube $Q$ which intersects
$\supp(\mu)$,
%(from the dyadic lattice described in next section),
we consider the closed ball $B_Q\!:=\! \overline B(z_Q,3\,\diam(Q))$, where
$z_Q$ and $\diam(Q)$ stand  for the center and diameter of $Q$,
respectively. Then we define
\begin{equation}\label{defalfa}
\alpha_\mu^n(Q) := \frac1{\ell(Q)^{n+1}}\,\inf_{c\geq0,L}
\,\dist_{B_Q}(\mu,\,c\HH^n_{|L}),
\end{equation}
where the infimum is taken over all the constants $c\geq0$ and all
the $n$-planes $L$. For convenience, if $Q$ does not intersect
$\supp(\mu)$, we set $\alpha^n_\mu(Q)=0$. To simplify notation, sometimes we
will also write $\alpha(Q)$ instead of $\alpha_\mu^n(Q)$.

 We denote by $c_Q$ and $L_Q$ the constant
and the $n$-plane that minimize $\dist_{B_Q}(\mu,\,\LL_L)$ (it is
easy to check that this minimum is attained). We also write
$\LL_Q:=c_Q\HH^n_{|L_Q}$, so that
$$\alpha_\mu^n(Q) = \frac1{\ell(Q)^{n+1}}\,\dist_{B_Q}(\mu,\,c_Q\HH^n_{|L_Q})
= \frac1{\ell(Q)^{n+1}}\,\dist_{B_Q}(\mu,\,\LL_Q).$$ Let us remark
that $c_Q$ and $L_Q$ (and so $\LL_Q$) may be not unique. Moreover,
we may (and will) assume that $L_Q\cap
{B_Q}\neq\varnothing$.

Recall that when $\mu$ is AD regular, one can construct some kind of
dyadic lattice of cubes adapted to the measure $\mu$. The cubes from
this lattice are not true cubes, although they play the role of
dyadic cubes with respect to $\mu$, in a sense. See \cite[Appendix
1]{David-LNM}, for example. The definitions of $\beta_p(Q)$ and
$\alpha(Q)$ are the same as above for this type of ``cubes''.

 In \cite{Tolsa-preprint} it is shown that $\beta_1(Q)\leq
C\alpha(Q)$ when $\mu$ is an AD regular $n$-dimensional measure and
$Q$ is a cube of the dyadic lattice associated to $\mu$. The
opposite inequality is false, in general.

We denote
$$\delta_\mu^n(x,r)=\frac{\mu(B(x,r))}{r^n},$$
and if $B=B(x,r)$, we set $\delta_\mu^n(B)=\delta_\mu^n(x,r)$.
Sometimes, to simplify notation we will write $\delta(x,r)$ instead
of $\delta_\mu^n(x,r)$.

% *************************************************************************
% *************************************************************************
% *************************************************************************
% *************************************************************************

\section{Auxiliary lemmas for the proof of Theorem
\ref{teolip}}\label{sec3}

\subsection{More notation and definitions}
\label{subdec}

Throughout Sections \ref{sec3}-\ref{sec6}, $\mu$ stands for the
measure described in the assumptions of Theorem \ref{teolip}. That
is, $\mu=g\,\HH^n_{\Gamma}$, where $\Gamma$ is the Lipschitz graph
$\{(x,y)\in\R^d:\,y=A(x)\}$. Observe that $\mu$ is AD regular.

 Recall that $\Pi$ is the projection
$(x_1,\ldots,x_n,\ldots,x_d)\mapsto (x_1,\ldots,x_n)$. We denote
$x_0 = \Pi(x) =(x_1,\ldots,x_n)$ (we identify $x_0\in\R^n$ with
$(x_0,0,\ldots,0)\in\R^d$) and, also, $x^\bot = \Pi^\bot(x)=
(x_{n+1},\ldots,x_d)$.

In the particular case of a Lipschitz graph and $\mu$ as above, the
construction of the dyadic lattice $\DD$ associated to $\mu$ is very
simple: let $\DD_0$ be the lattice of the usual dyadic cubes of
$\R^n$. A subset $Q\subset\Gamma$ is a cube from $\DD$ if and only
if it is of the form
$$Q=\Pi^{-1}(Q_0)\cap \Gamma$$
for some $Q_0\in\DD_0$. If $\ell(Q_0)=2^{-j}$ (where $\ell(\cdot)$
stands for side length), we set $\ell(Q)=2^{-j}$ and $Q\in\DD_j$. If
$z_{Q_0}$ is the center of $Q_0$, then we say that
$\Pi^{-1}(z_{Q_0})\cap\Gamma$ is the center of $Q$. The definition
of $\lambda Q$, for $\lambda>0$, is analogous.

Let $\psi$ be a non negative radial $\CC^\infty$ function such that
$\chi_{B(0,1/8)}\leq \psi\leq \chi_{B(0,1/4)}$. For each $j\in\Z$,
set $\psi_j(x) := \psi(2^jx)$ and $\vphi_j:=\psi_j-\psi_{j+1}$,  so
that each function $\vphi_j$ is non negative and supported on $B(0,
2^{-j-2})\setminus B(0,2^{-j-4})$, and moreover we have
$$\sum_{j\in\Z}\vphi_j(x) = 1 \qquad\mbox{for all $x\in\R^d\setminus \{0\}$.}$$
 We need to consider the following vectorial
kernels:
 \begin{equation}\label{defkj}
K_j(x)=\vphi_j(x_0) \frac{x}{|x|^{n+1}} \qquad j\in\Z,
\end{equation}
 and
 \begin{equation}\label{defkjbot}
\wt K_j(x)=\vphi_j(x_0) \frac{x}{|x_0|^{n+1}} \qquad j\in\Z,
\end{equation}
 for
$x\in\R^d$. The operators associated to $K_j$ and $\wt K_j$ are,
respectively,
$$R_j\mu(x) = \int K_j(x-y)\,d\mu(y),\qquad \wt R_j\mu(x) = \int \wt K_j(x-y)\,d\mu(y).$$
Notice that, formally,
$$R\mu(x) =\sum_{j\in\Z} R_j\mu(x).$$
Moreover, abusing notation sometimes we will write $R\mu$ instead of $\pv R\mu$. When $\mu$ is
like in Theorem \ref{teolip} this does not cause any trouble, since the $\mu$-a.e. existence of principal
values is a well known result.

Let us remark that, perhaps it would be more natural to replace
$\vphi_j(x)$ by $\vphi_j(x_0)$ in the definitions of the kernels
$K_j$ and $\wt K_j$ (like in \cite{Tolsa-preprint}). However, for
some of the calculations below the definitions above are more
convenient (although the choice of $\vphi_j(x)$ instead of
$\vphi_j(x_0)$ would also work with minor modifications and some
additional work).

We also denote by $K^i_j(x)$ and $\wt K^i_j(x)$ the $i$-th component
of $K_j(x)$ and $\wt K_j(x)$ respectively, and we set
$$
K_{j}^\bot(x) = \vphi_j(x_0)\frac{x^\bot}{|x|^{n+1}}$$ and
$$
\wt K_{j}^\bot(x) = \vphi_j(x_0)\frac{x^\bot}{|x_0|^{n+1}},$$ and we
denote by $R_j^\bot$ and $\wt R_j^\bot$ the corresponding operators
with kernels $K_j^\bot$ and $\wt K_j^\bot$.

% *************************************************************************
% *************************************************************************

\subsection{The upper estimate for the $L^2$ norm of Riesz transforms}

\begin{lemma} \label{lemupp}
Consider the $n$-dimensional Lipschitz graph $\Gamma:=\{(x,y)\in
\R^n \times \R^{d-n}:\,y=A(x)\}$. Suppose that $\|\nabla A\|_\infty\leq C_6$ and let $d\mu(z)=
g(z)\,d\HH^n_{|\Gamma}(z)$, where $g(\cdot)$ is a function such that
$C_1^{-1}\leq g(z)\leq C_1$ for all $z\in\Gamma$.
Then we have
$$\|\pv R\mu\|_{L^2(\mu)}\lesssim \|\nabla A\|_2 + \|g-1\|_2,$$
with constants depending on $C_6$ and $C_1$.
\end{lemma}

\begin{proof}
By Theorem 1.2 in \cite{Tolsa-preprint} we have
$$\|\pv R\mu\|_{L^2(\mu)}^2\lesssim\sum_{Q\in\DD}\alpha(Q)^2\mu(Q),$$
and by Theorem 1.1 and  Remark 4.1 in the same paper,
$$\sum_{Q\in\DD}\alpha(Q)^2\mu(Q)\lesssim \sum_{Q\in\DD}\beta_1(Q)^2\mu(Q) + \|g-1\|_2^2.$$
By \cite[Theorem 6]{Dorronsoro} we have
$$\sum_{Q\in\DD}\beta_1(Q)^2\mu(Q)\approx \|\nabla A\|_2^2,$$
and so the lemma follows.
\end{proof}

% *************************************************************************
% *************************************************************************

\subsection{Auxiliary lemmas for the lower estimate}
\label{sublemaux}

In the following lemma we collect a pair of trivial estimates. The
easy proof is left for the reader.

\begin{lemma}\label{lemtriv1}
Denote $\delta=2^{-j}$. For all $x\in\R^d$ and all $1\leq i \leq
n+1$, we have
$$|K_j^i(x)| \lesssim \frac{|x_i|}{\delta^{n+1}}\,\chi_{A(0,\delta/3,3\delta)},$$
and
$$|\nabla K_j^i(x)| \lesssim \frac{1}{\delta^{n+1}}\,\chi_{A(0,\delta/3,3\delta)}.$$
\end{lemma}

Notice that
$$\bigl| |x|-|x_0|\bigr|  \leq \frac{|x^\bot|^2}{|x|}.$$
 From this estimate and easy calculations, one gets

\begin{lemma}\label{lemtriv2}
Denote $\delta=2^{-j}$. For $x\in\R^d$ such that $|x|\approx |x_0|$,
and for $1\leq i \leq d$, we have
$$|K_j^i(x) - \wt K_j^i(x)| \lesssim \frac{|x_i| |x^\bot|^2}{\delta^{n+3}}\,\chi_{A(0,\delta/16,\delta)},$$
and
$$\bigl|\nabla\bigl(K_j^i - \wt K_j^i\bigr)(x)\bigr| \lesssim \frac{|x^\bot|^2}{\delta^{n+3}}\,\chi_{A(0,\delta/16,\delta)}.$$
\end{lemma}

The proof is left for the reader again.

\begin{lemma}  \label{lemfac1}
For all $j\in\Z$ and all $Q\in\DD_j$, we have
\begin{equation}\label{eqf1}
\int_Q |R_j\mu|^2\,d\mu \lesssim \bigl[\beta_2(Q)^2 +
\alpha(Q)^2\bigr]\mu(Q).
\end{equation}
Also, if $D_Q$ is the line that minimizes $\beta_1(Q)$ and
$$\beta_\infty(Q)\leq \ve_2 \quad\mbox{and}\quad \sin\measuredangle(D_0,D_{Q})\leq\ve_2,$$
with $\ve_2$ small enough, then
\begin{equation} \label{eqf2}
\int_Q |R_j\mu - \wt R_j\mu |^2\,d\mu \lesssim
\ve_2^4\bigl[\beta_2(Q)^2+ \alpha(Q)^2\bigl]\mu(Q).
\end{equation}
\end{lemma}

\begin{proof}
The estimate \rf{eqf1} has been proved in \cite[Lemma
5.1]{Tolsa-preprint}. The inequality \rf{eqf2} has a quite similar
proof. For completeness, we show the detailed arguments. Consider
the kernel $D_j=K_j-\wt K_j$, and let $T_{j}$ be the operator
associated to $D_{j}$.

Denote by $D_Q$ the line that minimizes $\beta_1(Q)$ and let $L_Q$
the one that minimizes $\alpha(Q)$.
 From the fact that
$\beta_1(Q)\lesssim\alpha(Q)$ it easily follows that
\begin{equation}\label{eq3*}
\dist_H(L_Q\cap B_Q,D_Q\cap B_Q)\lesssim \alpha(Q)\ell(Q),
\end{equation}
where $\dist_H$ stands for Hausdorff distance. Take $x\in Q\subset
\Gamma$. Consider the orthogonal projection $x'$ of $x$ onto $D_Q$.
Since we are assuming that $\beta_\infty(Q)$ is very small we have
$|x-x'|\ll\diam(Q)$ and then $\supp(D_j(x'-\cdot))\subset B_Q$.

First we will estimate $T_j\mu(x')$. Let $U$ be a thin tubular
neighborhood of $D_Q\cap B_Q$ of width $\leq C \ve_2\diam(Q)$
containing $\supp(\mu)\cap B_Q$ and denote $f(y) = D_j(x'-y)$.
Notice that for $y\in U\cap\supp(D_j(x'-y)$ we have $|x'-y|
\approx|x'_0-y_0|$, and so by Lemma \ref{lemtriv2}, for these $y$'s,
$$|\nabla f(y)|=\bigl|\nabla D_j(x'-y)\bigr| \lesssim \frac{|x'^\bot-y^\bot|^2}{\ell(Q)^{n+3}}.$$
We have
$$|x'^\bot - y^\bot|\lesssim \ell(Q) \bigl(\beta_\infty(Q) + \sin\meas(D_0,D_{Q})\bigr),$$
where $D_0$ stands for the $n$-plane $D_0=\R^n\times (0,\ldots0)$,
and so we get
\begin{equation}\label{eq23*}
\bigl|\nabla f(y)\bigr| \lesssim \frac{\ve_2^2}{\ell(Q)^{n+1}}.
\end{equation}

We extend $f_{|U\cap B_Q}$ to a function $\wt f$ supported on $B_Q$
with $\|\nabla \wt f\|_\infty\lesssim \ve_2^2/\ell(Q)^{n+1}.$
 Since
$K_j(\cdot)$ is odd and $x'\in D_Q$, we have $\int D_j(x'-y)\,d\HH^n_{|D_Q}(y)=0$,
and so
\begin{align*}
\biggl| \int D_j(x'-y)\,d\mu(y)\biggr| &=  \biggl| \int
D_j(x'-y)\,d\mu(y) - c_Q\int
D_j(x'-y)\,d\HH^n_{|D_Q}(y)\biggr| \\
& = \biggl| \int \wt f(y)\,d\mu(y) - c_Q\int \wt
f(y)\,d\HH^n_{|D_Q}(y)\biggr| \\
& \lesssim
\frac{\ve_2^2}{\ell(Q)^{n+1}}\,\dist_{B_Q}(\mu,c_Q\HH^n_{|D_Q}).
\end{align*}
In these estimates $c_Q$ stands for the constant minimizing the
definition of $\alpha(Q)$. By the definition of $\alpha(Q)$ and
\rf{eq3*} one easily gets
$$\dist_{B_Q}(\mu,c_Q\HH^n_{|D_Q})\lesssim\alpha(Q)\ell(Q)^{n+1}.$$ Thus,
$|T_j\mu(x')| \lesssim \ve_2^2\alpha(Q).$

Now we turn our attention to $T_j\mu(x)$. We have
$$|T_j\mu(x) - T_j\mu(x')|\lesssim
|x-x'|\sup_{\xi\in[x,x']}|\nabla T_j\mu(\xi)|.$$ By an estimate analogous to \rf{eq23*} we
have
$$|\nabla T_j\mu(\xi)| \leq \int\bigl|\nabla
D_j(\xi-y)\bigr|d\mu(y)\lesssim\frac{\ve_2^2}{\ell(Q)},$$ since
$|\xi-y|\approx|\xi_0-y_0|$ for $\xi\in[x,x']$ and
$y\in\supp(\mu)\cap\supp(D_j(\xi-\cdot))$.
 Therefore,
$$|T_j\mu(x) -
T_j\mu(x')|\lesssim\frac{\ve_2^2\,\dist(x,D_Q)}{\ell(Q)},$$ and so
$$|T_j\mu(x)| \lesssim \frac{\ve_2^2\dist(x,D_Q)}{\ell(Q)} + |T_j\mu(x')| \lesssim \ve_2^2\Bigl(\frac{\dist(x,D_Q)}{\ell(Q)} + \alpha(Q)\Bigr).$$
The lemma is a direct consequence of this estimate.
\end{proof}

From the preceding result we get the following.

\begin{lemma} \label{lemerrf}
For $j\in\Z$, let us denote
$$\beta_{2,j}(\Gamma)^2 := \sum_{Q\in\DD_j}\beta_2(Q)^2\mu(Q) \quad\mbox{and} \quad
\alpha_j(\Gamma)^2 := \sum_{Q\in\DD_j}\alpha(Q)^2\mu(Q).
$$
Suppose that
$$\beta_\infty(Q)\leq \ve_2 \quad\mbox{and}\quad \sin\measuredangle(D_0,D_{Q})\leq\ve_2,$$
where $D_Q$ is the line that minimizes $\beta_1(Q)$ and $\ve_2$ is
small enough.
 We have
$$\bigl|\langle R_j^\bot\mu,\,R_k^\bot\mu\rangle - \langle \wt R_j^\bot\mu,\,\wt
R_k^\bot\mu\rangle\bigr| \lesssim \ve_2^2\bigl(\beta_{2,j}(\Gamma)
+\alpha_j(\Gamma)\bigr) \bigl(\beta_{2,k}(\Gamma)
+\alpha_k(\Gamma)\bigr).$$
\end{lemma}

\begin{proof}
We set
\begin{align*}
\bigl|\langle R_j^\bot\mu,\,R_k^\bot\mu\rangle - & \langle \wt
R_j^\bot\mu,\,\wt R_k^\bot\mu\rangle\bigr| \\
& \leq \bigl|\langle R_j^\bot\mu - \wt
R_j^\bot\mu,\,R_k^\bot\mu\rangle\bigr| +\bigl| \langle \wt
R_j^\bot\mu,\,R_k^\bot\mu - \wt R_k^\bot\mu\rangle\bigr|\\
& \leq \|R_j^\bot\mu - \wt R_j^\bot\mu\|_2\|R_k^\bot\mu\|_2 + \|\wt
R_j^\bot\mu - R_j^\bot\mu\|_2\|R_k^\bot\mu - \wt R_k^\bot\mu\|_2\\
&\quad +\|R_j^\bot\mu\|_2\|R_k^\bot\mu - \wt R_k^\bot\mu\|_2.
\end{align*}
If we plug the estimates \rf{eqf1} and \rf{eqf2} into the preceding
inequality, the lemma follows.
\end{proof}

% *************************************************************************
% *************************************************************************
% *************************************************************************
% *************************************************************************

\section{The key Fourier estimate}

Consider
 the image measure $\sigma:=\Pi_\#\mu$
on $\R^n$ and set
$$H_{j}(x_0,y_0) = \vphi_j(x_0-y_0)\frac{A(x_0)  - A(y_0)}{|x_0-y_0|^{n+1}}.$$
 We have
\begin{align} \label{eqttt}
\langle \wt R_j^\bot\mu, \wt R_k^\bot\mu \rangle & = \iiint \wt
K_j^\bot(x,y) \wt K_k^\bot(x,z)
d\mu(x)d\mu(y)d\mu(z) \\
& = \iiint H_j(x_0,y_0) H_k(x_0,z_0)
d\sigma(x_0)d\sigma(y_0)d\sigma(z_0) =: I_0\nonumber
\end{align}
Below we will calculate $I_0$ using the Fourier transform in the
special case in which $\sigma$ coincides with the Lebesgue
$n$-dimensional measure on $\R^n$. This will allow us to prove
Theorem \ref{teolip} in this particular situation. The full theorem
will follow easily from this case.

%%%%%%%%%%%%%%%%%%%%%%%%%%%%%%%%%%%%%%%%%%%%%%%%%%%%%%%%%%%%%%%%%%%%%%%%%%%%%%%%%%%%%

\begin{lemma}\label{lemafourier}
Let us denote $\delta=2^{-j}$, $\ve= 2^{-k}$, and assume $\delta\leq \ve$. We have
\begin{align} \label{hj77}
0\leq\iiint_{(\R^n)^3} H_j(x,y)  H_k(x,z)\,dxdydz & \approx
\delta\ve \int_{|\xi|\leq1/\ve} |\wh A(\xi)|^2|\xi|^4 d\xi \\
&\quad +  \frac{\delta}{\ve}
 \int_{1/\ve\leq|\xi|\leq1/\delta}  |\wh A(\xi)|^2|\xi|^2 d\xi\nonumber \\
& \quad + \frac{1}{\delta\ve}\int_{|\xi|\geq1/\delta}  |\wh
A(\xi)|^2 d\xi. \nonumber
\end{align}
\end{lemma}

\begin{proof} For $x\in\R^n$, we denote $\eta(x) = \vphi(x)/|x|^{n+1}$. Notice
that
$$\frac{\vphi_j(x)}{|x|^{n+1}} = \frac1{\delta^{n+1}} \,\eta\Bigl(\frac x\delta\Bigr) =: \frac1\delta\eta_\delta(x),$$
and analogously for $\vphi_k(x)/|x|^{n+1}$. By the change of
variables $y=x+s$, $z=x+t$, and by Plancherel the triple integral on
the left hand side of \rf{hj77} equals
\begin{align*}
I_0 & := \iiint \biggl(
\vphi_j(x-y)\frac{A(x)-A(y)}{|x-y|^{n+1}}\biggr) \biggl(
\vphi_k(x-z)\frac{A(x)-A(z)}{|x-z|^{n+1}}\biggr)
dxdydz \\
& =\frac1{\delta\ve} \iiint \eta_\delta(s)\bigl(A(x)-A(x+s)\bigr)
\eta_\ve(t)\bigl(A(x)-A(x+t)\bigr)
dxdsdt \\
& = \frac1{\delta\ve} \iiint |\wh{A}(\xi)|^2 (1-e^{-2\pi i\xi
s})\eta_\delta(s)\overline{(1-e^{-2\pi i\xi t}) \eta_\ve(t)} d\xi dsdt.
\end{align*}
By Fubini, taking Fourier transform (for the $s$ and $t$ variables),
we get
$$I_0 = \frac1{\delta\ve}
\iiint |\wh{A}(\xi)|^2 \bigl(\wh \eta(0) -\wh \eta(\delta \xi)\bigr)
\overline{\bigl(\wh \eta(0) -\wh \eta(\ve \xi)\bigr)} d\xi.$$

Let
$$f_\delta(\xi):=\frac1\delta\bigl(\wh \eta(0) -\wh \eta(\delta \xi)\bigr).$$
It is easy to check that $f_\delta(\xi)$ is real and positive for $\xi\neq 0$\footnote{
This follows from the fact that
$$\wh \eta(0)= \int \eta(s)ds  > \int \cos(2\pi\xi s) \eta(s)ds= \wh \eta(\xi)$$
for all $\xi\neq0$, since $\eta$ is a non negative radial function
from ${\mathcal S}$.}. Moreover, using that $\wh \eta$ is radial and
$\wh\eta\in{\mathcal S}$, we get $f_\delta(\xi)\approx C\delta
|\xi|^2$ as $\xi\to0$, and $f_\delta(\xi)\approx C/\delta$ as
$|\xi|\to\infty$. So we infer that
$$f_\delta(\xi)\approx \delta|\xi|^2 \quad \mbox{if }\,|\xi|\leq
\frac1\delta,\quad \mbox{ and }\quad f_\delta(\xi)\approx
\frac1\delta \quad \mbox{if }\,|\xi|\geq \frac1\delta.$$
 Analogous
estimates hold for the corresponding function $f_\ve(\xi)$.
Therefore,
\begin{align}\label{i2t}
I_0 &\approx \delta\ve \int_{|\xi|\leq1/\ve} |\wh A(\xi)|^2|\xi|^4
d\xi + \frac{\delta}{\ve} \int_{1/\ve\leq|\xi|\leq1/\delta}  |\wh
A(\xi)|^2|\xi|^2 d\xi
\\ & \quad
+  \frac{1}{\delta\ve}\int_{|\xi|\geq1/\delta} |\wh A(\xi)|^2 d\xi.
\nonumber
\end{align}
\end{proof}

% ********************************************************************
% ********************************************************************
% ********************************************************************

\section{Proof of Theorem \ref{teolip} in the particular case $d\sigma\equiv dx$}

We will need the following result from \cite{Tolsa-preprint} (it is
not stated explicitly there, although it is proved in the paper):

\begin{theorem}\label{teopre}
Let $\mu$ be an $n$-dimensional $AD$ regular measure. For any
positive integer $N_0$, we have
\begin{equation}\label{eqn01}
\sum_{j,k:|j-k|> N_0} \bigl|\langle
R_j^\bot\mu,R_k^\bot\mu\rangle\bigr| \leq C
2^{-N_0/4}\sum_{Q\in\DD}\alpha(Q)^2\mu(Q).
\end{equation}
 Moreover, under the assumptions of Theorem \ref{teolip}, if $\Pi\#\mu = \rho(x)\,dx$, we have
\begin{equation}\label{eqn02}
\sum_{Q\in\DD}\alpha(Q)^2\mu(Q) \lesssim
\sum_{Q\in\DD}\beta_1(Q)^2\mu(Q) + \|\rho-1\|_2^2.
\end{equation}
\end{theorem}

Let us remark that in \cite{Tolsa-preprint} the preceding result has
been proved with $\vphi_j(x)$ replacing $\vphi_j(x_0)$ in the
definition of the kernel $K_j$ in \rf{defkj}. However, it is easy to
check that all the estimates of \cite{Tolsa-preprint} work with the
slightly different definition in \rf{defkj} when $\mu$ is supported
on a Lipschitz graph.

\begin{proof}[\bf Proof of Theorem \ref{teolip} in the particular case $d\sigma\equiv dx$]
By Lemma \ref{lemupp} we only need to prove the lower estimate $\|\pv R^\bot\mu\|_{L^2(\mu)}\gtrsim\|\nabla A\|_2.$
We set
$$\|R^\bot\mu\|_{L^2(\mu)}^2 = \sum_{j,k:|j-k|\leq N_0}\langle
R_j^\bot\mu,\,R_k^\bot\mu\rangle + \sum_{j,k:|j-k|>N_0}\langle
R_j^\bot\mu,\,R_k^\bot\mu\rangle =:S_1+ S_2.$$ In this identity
$R^\bot\mu$ can be understood either as the principal value or as an
$L^2(\mu)$ limit. We will show that if $\ve_0$ is small enough, then
$$S_1\approx\sum_{Q\in\DD}\beta_2(Q)^2\mu(Q)$$
(with constants depending on $N_0$), while $|S_2|\leq S_1/2$. The
theorem follows from these estimates.

The inequality
$$S_1\lesssim\sum_{Q\in\DD}\beta_2(Q)^2\mu(Q)$$
is a direct consequence of \rf{eqf1}, \rf{eqn02}, and the fact that $\rho\equiv1$.
 Now we consider the converse
estimate. We denote
$$T_j f(x) = \int_{\R^n} H_j(x,y)f(y)dy.$$
By \rf{eqttt} we have
$$ \langle \wt R_j^\bot\mu,\,\wt
R_k^\bot\mu\rangle = \langle T_j 1,\,T_k1\rangle_{\R^n}.$$
Then we set
\begin{align*}
\langle R_j^\bot\mu,\,R_k^\bot\mu\rangle & = \langle T_j
1,\,T_k1\rangle_{\R^n}
 + \bigl(\langle R_j^\bot\mu,\,R_k^\bot\mu\rangle -
 \langle
\wt R_j^\bot\mu,\,\wt R_k^\bot\mu\rangle\bigr)\\
& =: \langle T_j 1,\,T_k1\rangle_{\R^n} + E_{j,k}.
\end{align*}
 By Lemma \ref{lemafourier}, since $\langle T_j\mu,T_k\mu\rangle_{\R^n}\geq0$, we
have
\begin{align}\label{eqt1jk}
\sum_{j,k:|j-k|\leq N_0}\langle T_j 1,\,T_k1\rangle_{\R^n}& \geq
\sum_{j\in\Z} \|T_j1\|_2^2 \gtrsim \sum_{j\in\Z} \int_{2^{j-1}\leq |\xi|\leq 2^{j+1}}|\wh A(\xi)|^2|\xi|^2\,d\xi\\
& \approx\|\nabla A\|_2^2 \approx
\sum_{Q\in\DD}\beta_2(Q)^2\mu(Q).\nonumber
\end{align}

We consider now the terms $E_{j,k}$. Since $\|\nabla A\|_\infty\leq
\ve_0$, we infer that $\beta(Q)\lesssim\ve_0$, and then from Lemma
\ref{lemerrf} if $\ve_0$ is small enough we deduce
\begin{align*}
\sum_{j,k:|j-k|\leq N_0} |E_{j,k}| & \lesssim \ve_0^2
\sum_{j,k:|j-k|\leq N_0} \bigl(\beta_{2,j}(\Gamma)
+\alpha_j(\Gamma)\bigr) \bigl(\beta_{2,k}(\Gamma)
+\alpha_k(\Gamma)\bigr)\\
& \lesssim N_0 \ve_0^2\sum_{Q\in\DD} \bigl(\alpha(Q)^2 +
\beta_2(Q)^2\bigr)\mu(Q).
\end{align*}
From \rf{eqn02} we obtain
\begin{equation} \label{esttjk}
\sum_{j,k:|j-k|\leq N_0} |E_{j,k}| \lesssim N_0\ve_0^2
\sum_{Q\in\DD} \beta_2(Q)^2\mu(Q).
\end{equation}

By the estimates \rf{eqt1jk} and \rf{esttjk}, if
$\ve_0$ is small enough (for a given $N_0$), we infer that
\begin{equation}\label{eqss1}
 S_1\gtrsim\sum_{Q\in\DD}\beta_2(Q)^2\mu(Q).
\end{equation}

Finally we turn our attention to $S_2$. By Theorem \ref{teopre} we
have
 $$|S_2| \lesssim
2^{-N_0/4}\sum_{Q\in\DD}\beta_1(Q)^2\mu(Q).$$ Therefore, by
\rf{eqss1}, $S_2\leq C 2^{-N_0/4}S_1\leq S_1/2$ if $N_0$ is big
enough. We are done.
\end{proof}

% ********************************************************************
% ********************************************************************
% ********************************************************************

\section{Proof of Theorem \ref{teolip} in full
generality}\label{sec6}

\begin{lemma}\label{lemnormpet}
Consider the $n$-dimensional Lipschitz graph $\Gamma:=\{(x,y)\in
\R^n \times \R^{d-n}:\,y=A(x)\}$, with $\|\nabla A\|_\infty\leq C_8$,
and let $\mu$ be supported on $\Gamma $ such that
$d\Pi_\# \mu(x) = dx$.
Then $R^\bot_\mu$ is bounded in $L^2(\mu)$ with
$$\|R^\bot_\mu\|_{L^2(\mu),L^2(\mu)}\leq C_9\|\nabla A\|_\infty,$$
with $C_9$ depending only on $C_8$.
\end{lemma}

\begin{proof}
We think that this is essentially known. However, for completeness we give some details of the proof.
Consider the kernel
$$K(x,y)= \frac{A(x) - A(y)}{\bigl(|x-y|^2 +|A(x) - A(y)|^2\bigr)^{(n+1)/2}},$$
and the associated Calder\'on-Zygmund operator
$$Tf(x) = \int_{\R^n} K(x,y)\,f(x)\,dx,$$
for $f\in L^2(\R^n)$.
When $n=1=d-1$, we have the expansion
$$K(x,y) = \sum_{j=1}^\infty (-1)^j\frac{(A(x)-A(y))^{2j-1}}{|x-y|^{2j}} = \sum_{j=1}^\infty K_j(x,y),$$
and the corresponding associated operators are the Calder\'on
commutators $C_j$. It is well known that
$$\|C_j\|_{2,2}\leq C^{2j} \|\nabla A\|_\infty^{2j-1}$$
(see \cite[p.50]{David-LNM}, for example),
and so if $\|\nabla A\|_\infty$ is small enough the lemma follows.

For other $n$'s and $d$'s the result also holds. For example, it can be deduced from \cite{Tolsa-preprint}:
if $A$ is supported on a cube $Q$, then we have
$$\|R^{\bot}\mu\|_2\lesssim \|\nabla A\|_2\leq \|\nabla A\|_\infty\mu(Q)^{1/2}.$$
By a localization argument, one can prove that for any cube $P$,
$$\|R^{\bot}(\chi_P\mu)\|_2\lesssim \|\nabla A\|_\infty\mu(P)^{1/2},$$
and then by the $T1$ theorem the lemma follows (taking into account that the Calder\'on-Zygmund constants involved
in the kernel $K(x,y)$ are bounded above by $\|\nabla A\|_\infty$ too).
\end{proof}

\begin{remark} \label{remjac}
Consider the function $\wt A:\R^n\to\R^d$ given by $\wt A(x)=(x,A(x))$, where $A$ is the Lipschitz function that defines the Lipschitz graph $\Gamma$. Notice that the density function $\rho(x)$ such that $\Pi_\#\mu = \rho(x)dx$
is given by
$$\rho(x) = g(x)\,J\wt A(x),$$
where $J\wt A(x)$ stands for the $n$-dimensional Jacobian of $\wt A$. Recall that
$$J\wt A(x) = \Bigl(\sum_{B}(\det B)^2\Bigr)^{1/2},$$
where the sum runs over all the $n\times n$ submatrices $B$ of $D\wt
A(x)$, the differential map of $\wt A$ at $x$ (see \cite[p.\
24]{Morgan}, for example). Then it is easy to check that
$$(J\wt A(x))^2 = 1 + e(x),$$
with $$|e(x)| \lesssim \sup_{i,j}|\partial_iA_j(x)|^2$$
(in fact, $e(x) = \sum_{i,j} (\partial_i A_j(x))^2 + \ldots$, where ``$\ldots$" stands for some terms which involve higher
order products of derivatives of $A$). So we also have
$$J\wt A(x) = 1 + e_0(x),$$
with $$|e_0(x)| \lesssim \sup_{i,j}|\partial_iA_j(x)|^2.$$
As a consequence,
$$|\rho(x) -1| = |g(x)(1+e_0(x)) - 1| \leq |g(x)-1| + C|e_0(x)|.$$
Observe  that $\|e_0\|_\infty\lesssim \|\nabla
A\|_\infty^2\leq\ve_0^2$ and $\|e_0\|_2\lesssim \|\nabla
A\|_\infty\|\nabla A\|_2$. Then the assumptions of Theorem
\ref{teolip} ensure that
\begin{equation}\label{eqaaa2}
\|\rho-1\|_2\leq \|g-1\|_2 + C\|e_0\|_2\leq C \|\nabla A\|_2.
\end{equation}
\end{remark}

\vv
\begin{proof}[\bf Proof of Theorem \ref{teolip}]
Recall that we only need to prove the lower estimate $\|\pv R^\bot\mu\|_{L^2(\mu)}\gtrsim\|\nabla A\|_2.$
Consider the measure $\mu_0$ supported on $\Gamma$ such that
$\Pi_\#\mu_0 = dx$.
Recall that
\begin{equation}\label{eqmu0}
\|R^\bot\mu_0\|_{L^2(\mu_0)}\approx\|\nabla A\|_2.
\end{equation}
Since
$$\Pi_\#\mu =g(x)\,J\wt A(x)\,dx=:\rho(x)\,dx, \qquad x\in\R^n,$$
it turns out that
$$\rho(\Pi(x))\,d\mu_0(x)= d\mu(x),\qquad x\in\R^d.$$
We denote $h(x) = \rho(\Pi(x))$, and so we have
$$d\mu(x) - d\mu_0(x) = (h(x)-1)\,d\mu_0(x),$$
with $\|h-1\|_{L^2(\mu_0)}\lesssim \|\nabla A\|_2,$ by \rf{eqaaa2}.
So, from Lemma \ref{lemnormpet} we deduce
\begin{align*}
\bigl|\|R^\bot\mu\|_{L^2(\mu_0)} - \|R^\bot\mu_0\|_{L^2(\mu_0)}\bigr| & \leq
\|R^\bot\mu - R^\bot\mu_0\|_{L^2(\mu_0)} \\ &= \|R^\bot((h-1)\,d\mu_0)\|_{L^2(\mu_0)} \\
& \lesssim \|\nabla A\|_\infty
\|h-1\|_{L^2(\mu_0)}\leq \ve_0\|\nabla A\|_2.
\end{align*}
If $\ve_0$ is small enough, from \rf{eqmu0} we infer that
$$\|R^\bot\mu\|_{L^2(\mu_0)} \approx \|R^\bot\mu_0\|_{L^2(\mu_0)}\approx \|\nabla A\|_2,$$
which implies that
$$\|R^\bot\mu\|_{L^2(\mu)} \approx \|\nabla A\|_2,$$
since $g(x)\approx h(x)\approx 1$ for all $x$.
\end{proof}

% ********************************************************************
% ********************************************************************
% ********************************************************************

\section{The Main Lemma for the proof of Theorem \ref{teovp}}\label{sec7}

This and the remaining sections are devoted to the proof of Theorem \ref{teovp}.

For $\ve>0$ we denote
$$\wt R_\ve\mu(x) \int \frac{x-y}{\bigl(|x-y|^2+\ve^2\bigr)^{(n+1)/2}}\,d\mu(y),$$
and also
$$\wh R_\ve\mu(x) = \int \psi(\ve^{-1}(x-y)) \frac{x-y}{|x-y|^{n+1}}\,d\mu(y),$$
where $\psi$ is a $\CC^\infty$ radial function such that $\chi_{\R^d\setminus B(0,1)}\leq\psi\leq
\chi_{\R^d\setminus B(0,1/2)}$.
We also set
$$\wt R_{\ve_1,\ve_2}\mu(x) = \wt R_{\ve_1}\mu(x) -
\wt R_{\ve_2}\mu(x),$$
and
$$\wh R_{\ve_1,\ve_2}\mu(x) = \wh R_{\ve_1}\mu(x) - \wh R_{\ve_2}\mu(x).$$
It is easy to check that if ${\rm p.v.}R\mu(x)$ exists for some $x\in\R^d$, then
$$\lim_{\ve\to0} \wt R_\ve\mu(x) = \lim_{\ve\to0} \wh R_\ve\mu(x) = \lim_{\ve\to0}R_\ve\mu(x).$$
(Hint: write $\wt R_\ve\mu(x)$ and $\wh R_\ve\mu(x)$ as a convex
combination of $R_\ve\mu(x)$, $\ve>0$. We also denote
$c_n=\LL^n(B_n(0,1))$, where $\LL^n$ stands for the $n$-dimensional Lebesgue measure.

Theorem \ref{teovp} is a consequence of the following result.

\begin{mlemma}\label{mlem}
Let $\mu$ be a finite Borel measure on $\R^d$. Let $B_0=\overline
B(x_0,r_0)$ be a closed ball such that there exists a compact subset
$F\subset 10B_0$, with $x_0\in F$, which satisfies
\begin{itemize}
\item[(a)] $\mu(8B_0) =c_n 8^nr_0^n$ and $\mu(10B_0\setminus F)\leq \delta_1 \mu(B_0)$,
\item[(b)] $\mu(B(x,r))\leq M_1r^n$ for all $x\in F,r>0$, and $\mu(B(x,r))\leq c_n(1+\delta_1)r^n$ for all $x\in F$
and $0<r\leq 100r_0$,
\item[(c)] $\|R_\mu\|_{L^2(\mu|F),L^2(\mu|F)}\leq M_2$,
\item[(d)] $|\wt R_{\ve_1,\ve_2}\mu(x)| + |\wh R_{\ve_1,\ve_2}\mu(x)| \leq \delta_2$ for all $x\in F$ and
$0<\ve_1<\ve_2\leq \delta_2^{-2} r_0$.
\end{itemize}
If $\delta_1,\delta_2$ are small enough, with
$\delta_1=\delta_1(M_2)$ and $\delta_2=\delta_2(M_1,M_2)$, then
there exists an $n$-dimensional Lipschitz graph $\Gamma$ such that
$$\mu(\Gamma\cap F\cap B_0)\geq \frac9{10}\,c_nr_0^n.$$
\end{mlemma}

Let us remark that the Lipschitz constant of the graph $\Gamma$
depends on the constants $M_1,M_2$ and $\delta_1,\delta_2$, and
tends to $0$ as $\delta_1+\delta_2\to0$, for fixed $M_1,M_2$.

\begin{proof}[\bf Proof of Theorem \ref{teovp} using Main Lemma
\ref{mlem}]

Consider an arbitrary subset $\wt E\subset E$. Given $\delta>0$, for each
$i\in\Z$ set
$$E_i=\{x\in \wt E:\,(1+\delta)^i\leq\Theta^{n,*}_\mu(x)<
(1+\delta)^{i+1}\},$$
 so that $\mu(\wt E\setminus\bigcup_i E_i)=0$. For $j\geq1$, denote
$$E_{i,j}= \{x\in E_i:\,\delta_\mu^n(x,r)\leq (1+\delta)^{i+2}\,\mbox{ if } 0<r\leq 1/j,\}.$$
Notice that for all $x\in E_{i,j}$ we have
$$\mu(B(x,r))\leq M_{i,j}r^n\qquad\mbox{
for all $r>0$ and some fixed $M_{i,j}$.}$$
From the fact that $R_*\mu(x)<\infty$ on $E$, arguing as in \cite{Tolsa-pams}, we can split each set
$E_{i,j}$ as
$$E_{i,j}= \bigcup_{k\geq 1} E_{i,j,k},$$
so that, for each $k$,
$$\|R_{\mu_{|E_{i,j,k}}}\|_{L^2(\mu_{|E_{i,j,k}}),L^2(\mu_{|E_{i,j,k}})}\leq
k.$$
 Given any constant $\ve_0>0$, for each $m\geq1$ we set
$$E_{i,j,k,m} = \Bigl\{x\in E_{i,j,k}:\sup_{0<\ve_1<\ve_2\leq
1/m}\bigl(|\wt R_{\ve_1,\ve_2}\mu(x)|+ |\wh
R_{\ve_1,\ve_2}\mu(x)|\bigr)\leq \ve_0\Bigr\}.$$ It is clear that
$$\wt E = \bigcup_{i,j,k,m} E_{i,j,k,m}.$$

Consider $\wt E_{i,j,k,m}\subset E_{i,j,k,m}$ such that $\wt
E_{i,j,k,m}\cap \wt E_{i',j',k',m'}=\varnothing$ if
$(i,j,k,m)\neq(i',j',k',m')$ and we still have
$$\wt E = \bigcup_{i,j,k,m} \wt E_{i,j,k,m}.$$
 For each density point $x$ of $\wt E_{i,j,k,m}$ consider a ball $B_x=B(x,r_x)$
with radius $0<r_x\leq \min(1/(100j),\,\ve_0/m)$ such that
$$\mu(B_x\setminus \wt E_{i,j,k,m}) \leq \delta\,\mu(\wt E_{i,j,k,m})$$
and
$$(1+\delta)^{i-1}\leq \delta_\mu^n(x,r_x) \leq (1+\delta)^{i+2}.$$
 If we take $\delta$ and $\ve_0$ small enough, we set $F:=\wt E_{i,j,k,m}$, and we apply
Main Lemma \ref{mlem} to the measure $\frac{c_n\,r_x^n}{\mu(B_x)}\,\mu$
and to the ball $B_0=\frac18 B_x$, we infer the
 existence of a Lipschitz graph such as the one described in the
 Main Lemma.
 If we consider a Vitali type covering with a family of disjoint balls $B_{x_i}$
 we deduce that there exists a rectifiable subset $F_{i,j,k,m}\subset
 \wt E_{i,j,k,m}$ with $\mu(F_{i,j,k,m})\geq
 \frac9{10}\,\mu(\wt E_{i,j,k,m})$.
We set $\wt F :=\bigcup_{i,j,k,m}F_{i,j,k,m}$, and then we have
 $$\mu(\wt F) \geq
 \frac9{8^n10}\,\mu(\wt E).$$
 It is easy to check that this
 implies that $E$ is rectifiable.
\end{proof}

\vv
The remaining sections of the paper are devoted to the proof of Main Lemma \ref{mlem}.

% ********************************************************************
% ********************************************************************
% ********************************************************************

\section{Flatness of $\mu$ when the Riesz transforms are small}

We set
$$P(x,\ve) = \int \frac{\ve}{\bigl(|x-y|^2+\ve^2\bigr)^{(n+1)/2}}\,d\mu(y)$$
and
$$P_2(x,\ve)  = \int \frac{\ve^3}{\bigl(|y|^2+\ve^2\bigl)^{(n+3)/2}}
 d\mu(y).$$

\begin{lemma}\label{lemtaylor}
Let $\mu$ be a Borel measure on $\R^d$.
Consider $\ve>0$ and $x\in\R^{d}$ such that $|x|\leq \ve/4$. We have
$$\wt R_\ve\mu(x) - \wt R_\ve\mu(0) = T(x) + E(x),$$
with
\begin{equation}\label{eqte2}
T(x) = \int\frac{\bigl(|y|^2+\ve^2\bigr)x - (n+1)(x\cdot y)y}{\bigl(|y|^2+\ve^2)^{(n+3)/2}}\,d\mu(y),
\end{equation}
and
$$|E(x)|\leq C_{10}\,\frac{|x|^2}{\ve^2}\,P(0,\ve).$$
\end{lemma}

\begin{proof} The arguments are analogous to the ones of Lemma 5.1 in \cite{Tolsa-gafa} for the Cauchy transform.
We will show the details for completeness.

The Taylor expansion of the function $1/(s+\ve^2)^{(n+1)/2}$ at $s_0$ is
\begin{align*}
\frac1{(s+\ve^2)^{(n+1)/2}} & = \frac1{(s_0+\ve^2)^{(n+1)/2}} - \frac{n+1}{2(s_0+\ve^2)^{(n+3)/2}}\,(s-s_0) \\
&\quad + \frac{(n+1)(n+3)}{8(\xi + \ve^2)^{(n+5)/2}}\,(s-s_0)^2,
\end{align*}
where $\xi\in[s_0,s]$.
If we set $s_0=|y|^2$, $s=|x-y|^2$, and we multiply by $x-y$, we obtain
\begin{align*}
\frac{x-y}{(|x-y|^2+\ve^2)^{(n+1)/2}} = & \frac{x-y}{(|y|^2+\ve^2)^{(n+1)/2}} -
\frac{\frac{n+1}2\,(x-y)}{(|y|^2 + \ve^2)^{(n+3)/2}}\,(|x|^2-2x\cdot y) \\ &\mbox{} +
\frac{(n+1)(n+3)(x-y)}{8(\xi_{x,y} + \ve^2)^{(n+5)/2}}\,(|x|^2-2x\cdot y)^2,
\end{align*}
where $\xi_{x,y}\in[|y|^2,|x-y|^2]$.
If we integrate with respect to $d\mu(y)$, we get
$$\wt R_\ve\mu(x) = \wt R_\ve\mu(0) + T(x) + E(x),$$
with
\begin{align*}
E(x) & = \frac{n+1}2\int \frac{|x|^2(x-y) + 2(x\cdot y)x}{(|y|^2 + \ve^2)^{(n+3)/2}}\,d\mu(y) \\
&\quad +
\int \frac{(n+1)(n+3)(x-y)}{8(\xi_{x,y} + \ve^2)^{(n+5)/2}}\,(|x|^2-2x\cdot y)^2\,d\mu(y) =: E_1(x) + E_2(x).
\end{align*}
To estimate $E_1(x)$, from $|x|\leq\ve/4$ and $\bigl||x|^2(x-y) + 2(x\cdot y)x\bigr|\leq C|x|^2(|y|+\ve)$ we deduce
$$|E_1(x)| \lesssim \int \frac{|x|^2}{(|y|^2 + \ve^2)^{(n+2)/2}}\,d\mu(y) \leq \frac{|x|^2}{\ve^2}\,P(0,\ve).$$
For $E_2(x)$ we take into account that $\xi_{x,y}+\ve^2\approx |y|^2+\ve^2$ and, again, that $|x|\leq\ve/4$.
Then,
\begin{align*}
|E_2(x)| & \lesssim
|x|^2 \int \frac{(|x|+|y|)^3}{(\bigl|y|^2 + \ve^2\bigr)^{(n+5)/2}}\,\,d\mu(y)
\lesssim |x|^2 \int \frac{1}{(|y|^2 + \ve^2)^{(n+2)/2}}\,\,d\mu(y) \\ & \leq
\frac{|x|^2}{\ve^2}\,P(0,\ve).
\end{align*}
\end{proof}

We will need the following result. See \cite[Lemma 2.8]{Leger-tesi} for the proof, for example.

\begin{lemma} \label{lemli}
Let $\mu$ be a Borel measure on $\R^d$. Suppose that $\mu(B(x,r))\leq r^n$ for all $x\in\R^d$. Let
$B(y,t)$ be a ball such that $\delta(y,t)\geq C_{11}^{-1}$. Then there are $n+1$ balls $\Delta_0,\ldots,\Delta_n$ centered at
$\supp(\mu)\cap B(y,t)$ with radius $t/C_{12}$ such that $\delta(B_i)\geq C_{13}^{-1}$ and for all
$(x_0,\ldots,x_n)\in \Delta_0\times\ldots\times \Delta_n$ we have
\begin{equation}\label{eqvol}
{\rm vol}^n((x_0,\ldots,x_n))\geq \frac{t^n}{C_{14}},
\end{equation}
where ${\rm vol}^n((x_0,\ldots,x_n))$ denotes the $n$-volume of the $n$-simplex with vertices $x_0,\ldots,x_n$.
\end{lemma}

The arguments for the following lemma are very similar to the ones
of \cite[Lemma 7.4]{Tolsa-preprint}. We will show again the detailed proof
for the sake of completeness.

\begin{lemma}\label{lemcas1}
Let $B(y,t)$ and let $x_0,\ldots,x_n\in B(y,t)$ satisfy \rf{eqvol}.
Then any point $x_{n+1}\in B(y,3t)$ satisfies
$$\dist(x_{n+1},L) \lesssim \frac\ve{P_2(x_0,\ve)} \sum_{j=1}^{n+1} |\wt R_\ve\mu(x_j)- \wt R_\ve\mu(x_0)| +
\frac{P(x_0,\ve)}{P_2(x_0,\ve)}\,\frac{t^2}{\ve},$$
where $L$ is the $n$-plane passing through $x_0,\ldots,x_n$.
\end{lemma}

\begin{proof}
We only have to consider the case $\ve>t$ and moreover, without loss of generality, we assume that $x_0=0$.
We denote  by $z$ the orthogonal projection of $x_{n+1}$ onto $L$.
Then by Lemma \ref{lemtaylor} we have
\begin{equation}\label{eqlf1}
|T(x_j)|\lesssim |\wt R_\ve\mu(x_j)- \wt R_\ve\mu(x_0)| +  \frac{t^2}{\ve^2}\,P(0,\ve).
\end{equation}
for $j=1,\ldots,n+1$.
Let $e_1,\ldots,e_n$ be an orthonormal basis of $L$, and set $e_{n+1} =
(x_{n+1}-z)/|x_{n+1}-z|$ (we suppose that $x_{n+1}\not\in L$), so that $e_{n+1}$ is
a unitary vector orthogonal to $L$.
Since the points $x_j$, $j=1,\ldots,n$ are linearly independent with ``good constants'' (i.e. they satisfy \rf{eqvol}) we get
$$|T(e_i)|\lesssim \frac1t \sum_{j= 1}^n|T(x_j)|
\lesssim  \sum_{j= 1}^n
|\wt R_\ve\mu(x_j)- \wt R_\ve\mu(x_0)| +  \frac{t}{\ve^2}\,P(0,\ve).
$$
for $i= 1,\ldots,n$. Also, since $z\in L$ and $|z|\lesssim t$, we have
$|T(z)|\lesssim \sum_{j= 1}^n|T(x_j)|,$
and so by~\rf{eqlf1},
\begin{align*}
|T(e_{n+1})| & = \frac1{\dist(x_{n+1},L)}\,|T(z-x_{n+1})|  \\
& \lesssim \frac1{\dist(x_{n+1},L)}\Bigl(
 \sum_{j= 1}^{n+1} |\wt R_\ve\mu(x_j)- \wt R_\ve\mu(x_0)| + \frac{t^2}{\ve^2}\,P(0,\ve)\Bigr).
\end{align*}
Therefore,
\begin{equation}\label{eqt5}
\Bigl|\sum_{j=1}^{n+1}T(e_j)\cdot e_j\Bigr| \lesssim \frac1{\dist(x_{n+1},L)}\,\Bigl(
 \sum_{j= 1}^{n+1} |\wt R_\ve\mu(x_j)- \wt R_\ve\mu(x_0)| + \frac{t^2}{\ve^2}\,P(0,\ve)\Bigr).
\end{equation}

On the other hand, from the definition of $T$ in \rf{eqte2}, if we denote $y_{(i)} = y\cdot e_i$, we get
\begin{align}\label{eqt55}
\sum_{j=1}^{n+1}T(e_j)\cdot e_j & =
\int \frac{(n+1)\bigl(|y|^2+ \ve^2\bigr) - (n+1) \sum_{j=1}^{n+1}y_{(j)}^2}{\bigl(|y|^2+\ve^2\bigl)^{(n+3)/2}}
 d\mu(y)\\
&= (n+1)
\int \frac{\ve^2+\sum_{j={n+2}}^{d}y_{(j)}^2 }{\bigl(|y|^2+\ve^2\bigl)^{(n+3)/2}}
 d\mu(y)\geq \frac{(n+1)}\ve\, P_2(0,\ve). \nonumber
\end{align}
The lemma follows from \rf{eqt5} and \rf{eqt55}.
\end{proof}

\begin{lemma}\label{lempoi}
Let $\mu$ be a Borel measure on $\R^d$ and $B(x,r)$ such that
$$\mu(B(x,r))\geq C_{15}^{-1}\,r^n,\qquad \mu(B(x,t))\leq M t^n \mbox{ for all $t\geq r$}.$$
Then there exists $r_1$ with $r\leq r_1\leq C_{16} r$, with $C_{16}$ depending on $C_{15}$ and $M$, such that
$$P(x,r_1)\leq 2^{n+4}\delta(x,r_1) \qquad\mbox{and} \qquad \delta(x,r_1)\geq\delta(x,r).$$
\end{lemma}

\begin{proof}

To simplify notation we set $a= 2^{n+4}$. The lemma follows from the following:

\begin{claim*}
Under the assumptions of the lemma, either $P(x,r)\leq a\delta(x,r)$ or there exists some $t$ with $r\leq t\leq C_{18}r$ (with $C_{18}$ depending on $C_{15}$ and $M$) such that $\delta(x,t)\geq 8\, \delta(x,r)$.
\end{claim*}

Suppose that the above statement holds. If $P(x,r)> a\delta(x,r)$, then there exists $s_1$ with
$r< s_1\leq C_{18} r$ such that $\delta(x,s_1)\geq 8\, \delta(x,r)$.

%We apply again the claim with $s_1$
%replacing $r$, and so either $\delta(x,s_1)\leq a\, P(x_0,s_1)$, or we find $s_2$ with $s_1< s_2\leq C_{18} s_1$.
By repeated application of the claim, we deduce that
either there exists a sequence $s_1,s_2,s_3,\ldots, s_m$ such that
\begin{equation}\label{eqseq}
\delta(x,s_m)\geq 8 \, \delta(x,s_{m-1})\geq \ldots \geq 8^{m-1}\, \delta(x,s_1)
\geq 8^m\,C_{15}^{-1},
\end{equation}\label{opcio2}
or
\begin{equation}
\mbox{there exists some $s_j$, $1\leq j \leq m-1$, such that $P(x,s_j)\leq a\delta(x,s_j)$.}
\end{equation}
The statement \rf{eqseq} fails for $m$ big enough since $\delta(x,s_j)\leq M$ for all $j$. Thus
\rf{opcio2} holds for some $j$ big enough, and so the lemma follows by choosing the minimal such $j$.

To prove the claim we set
\begin{align*}
P(x,r) & = \biggl(
\int_{|x-y|\leq r}  +
\sum_{k\geq1}
\int_{2^{k-1}r<|x-y|\leq 2^{k}r} \biggr) \frac{r}{\bigl(|x-y|^2+r^2\bigr)^{(n+1)/2}}\,d\mu(y) \\
& \leq \frac{\mu(B(x,r))}{r^n} + \sum_{k\geq 1}\frac{r}{(2^{k-1}r)^{n+1}}\,\mu(B(x,2^{k}r))\\
& \leq \delta(x,r) + \sum_{k=1}^N 2^{n+1-k}\delta(x,2^{k}r) + M \sum_{k\geq N+1} 2^{n+1-k}\\
&=  \delta(x,r) + \sum_{k=1}^N 2^{n+1-k}\delta(x,2^{k}r) + M  2^{n+1-N}.
\end{align*}
Since $P(x,r)\geq a \delta(x,r)$ we infer that
$$(a-1) \delta(x,r)\leq \sum_{k=1}^N 2^{n-k}\delta(x,2^{k}r) + M  2^{n+1-N}.$$
For $N$ big enough we have $M  2^{n+1-N}\leq C_{15}^{-1} \leq \delta(x,r)$, and so
$$(a-2)\delta(x,r)\leq 2^{n+1}\sum_{k=1}^N 2^{-k}\delta(x,2^{k}r),$$
which implies that there exists some $k\in [1,N]$ such that
$$\delta(x,2^{k}r)\geq 2^{-n-1}(a-2)\delta(x,r)\geq 8\,\delta(x,r)$$
(recall that $a= 2^{n+4}$).
\end{proof}

\begin{lemma} \label{lemflat}
Let $\mu$ be a Borel measure on $\R^d$, $F\subset \R^d$ and $B=B(x,r)$ such that
\begin{equation}\label{eqdel1}
|\wt R_\ve\mu(y) - \wt R_\ve\mu(z)|\leq \delta\qquad \mbox{for all $y,z\in F\cap 3B$ and
$r\leq\ve\leq \delta^{-1}r$,}
\end{equation}
and
$$\mu(F\cap B)\geq C_{15}^{-1}\,r^n,\qquad \mu(B(x,t))\leq M t^n \mbox{ for all $t\geq r$}.$$
Then we have
$$\beta_{\infty,F}(B)\leq \ve_1,$$
with $\ve_1$ depending on $C_{15},\delta,M$, and $\ve_1\to0$ as $\delta\to 0$ for each fixed $C_{15},M$.
\end{lemma}

\begin{proof}
Let $\Delta_0,\ldots,\Delta_n$ be balls of radius $t$ like the ones in Lemma \ref{lemli} with $C^{-1}r\leq t\leq r$ and
$\mu(F\cap \Delta_i)\gtrsim t$ (we apply Lemma \ref{lemli} to $\mu_{|F\cap B}$).
Consider $z_i\in F\cap \Delta_i$ for each $i=0,\ldots,n$. Given any $\ell$ with $r\leq \ell\leq \delta^{-1} r$,
by \rf{eqdel1} and Lemma \ref{lemcas1}, for any $y\in F\cap 3B$
we have
\begin{equation}\label{eqrd35}
\dist(y,L) \leq C \frac\ell{P_2(x,\ell)} \delta +
\frac{C\,P(x,\ell)}{P_2(x,\ell)}\,\frac{r^2}{\ell},
\end{equation}
where $L$ is the $n$-plane passing through $z_0,\ldots,z_n$.

Given $\ve_1>0$, take $s\geq r$ such that
$$\frac{C\,r^2}{s}\leq\frac{\ve_1r}2.$$ Notice that $\delta(x,s)\geq C(\ve_1)\delta(x,r)$.
By Lemma \ref{lempoi},
we can choose $\ell\geq s$ such that $s\leq \ell\leq C_{16} s$ (with $C_{16}$ depending on $\ve_1$) and
$$P(x,\ell)\leq 2^{n+4}\delta(x,\ell) \qquad\mbox{and} \qquad P_2(x,\ell)\geq C^{-1} \delta(x,\ell)\geq C(\ve_1)^{-1}\delta(x,r).$$
 Moreover, if $\delta$ is small enough then we also have $\ell\leq\delta^{-1}r$, so that \rf{eqrd35} holds, and then we deduce that
$$\dist(y,L) \lesssim C(\ve_1) \delta \ell+ \frac{\ve_1 r}2 \leq C \ve_1 r,$$
if $\delta\ll C(\ve_1)^{-1}$.
\end{proof}

% ********************************************************************
% ********************************************************************

\section{Construction of the Lipschitz graph for the proof of Main Lemma \ref{mlem}}

% ********************************************************************
% ********************************************************************

\subsection{Léger's theorem}

To construct the Lipschitz graph $\Gamma$ we will follow quite
closely the arguments of \cite{Leger}. Recall that in this paper the
author proves that if $E\subset\R^d$ has finite length and finite
curvature, then $E$ is rectifiable (i.e.\ $1$-rectifiable). A more
precise result is the following (see \cite[Proposition 1.1]{Leger}:

\begin{theorem}\label{teoleger}
For any constant $C_{17}\geq 10$, there exists a number $\eta>0$ such
that if $\sigma$ is a Borel measure on $\R^d$ verifying
\begin{itemize}
\item $\sigma(B(0,2))\geq1$, $\supp\sigma\subset B(0,2)$,
\item for any ball $B$, $\sigma(B)\leq C_{19}\diam(B)$,
\item $c^2(\sigma)\leq \eta$,
\end{itemize}
then there exists a Lipschitz graph $\Gamma$ such that
$\sigma(\Gamma)\geq \frac{99}{100}\,\sigma(\R^n)$.
\end{theorem}

Let us remark that, although L\'eger's theorem is a $1$-dimensional
result, it easily generalizes to higher dimensions, as the author
claims in \cite{Leger}.

Instead of an estimate on the curvature of $\mu$, to prove the Main Lemma~\ref{mlem} we will use
$L^2(\mu)$ estimates of Riesz transforms (by means of Theorem \ref{teolip}).

% ********************************************************************
% ********************************************************************

\subsection{The stopping regions for the construction of the Lipschitz graph}

In the rest of the paper we assume that $\mu$, $B_0$ and $F$ satisfy
the assumptions of Main Lemma \ref{mlem}.

Notice that by Lemma \ref{lemflat} we know that there exists some
$n$-plane $D_0$ such that
$$\dist(x,D_0)\leq C\delta r_0 \quad \mbox{ for all $x\in F$.}$$
Without loss of generality we will assume that $D_0=
\R^n\times\{(0,\ldots,0)\}\equiv\R^n$.

%We consider the orthogonal projection $\Pi$ from $10B_0\cap\supp(\mu)$ onto $\R^n$.

As stated above, to construct the Lipschitz graph, we follow very closely the arguments from
\cite{Leger}. First we need to define a family of stopping time
regions, which are the same as the ones defined in \cite[Subsection 3.1]{Leger}.
Given positive constants $\delta_0,\ve,\alpha$ to be fixed below, we set
 $$S_{total}= \left\{(x,t)\in (F\cap B_0)\times (0,8r_0),\begin{array}{ll}
 \mbox{(i)}& \delta_{F}(x,t)\geq \frac12\delta_0\\
 \mbox{(ii)}& \beta_{1,F}(x,t)<2\ve \\
 \mbox{(iii)}& \exists D_{x,t} \mbox{ s.t. }\left\{\begin{array}{l}
\beta_{1,F}^{D_{x,t}}(x,t)\leq2\ve,\mbox{ and} \\
\meas(D_{x,t},D_0)\leq\alpha
\end{array}\right.
 \end{array} \!\!\!\!\!\right\}.
 $$
In the definition above to simplify notation we have denoted $\delta_F(x,r)\equiv\delta_{\mu_{|F}}(x,r)$ and
$\beta_{1,F}(x,r)\equiv\beta_{1,\mu_{|F}}(B(x,r))$. Also.
$D_{x,t}$ are $n$-planes depending on $x$ and $t$ and
$$\beta_{1,F}^{D_{x,t}}(x,t) =\frac1{t^n}\int_{y\in F:|x-y|\leq 3t}\frac{\dist(y,D_{x,t})}t\,d\mu(y).$$

Let us remark that $\delta_0,\ve,\alpha$ will be chosen so that $0<\ve\ll\alpha\ll\delta_0\ll1$.

For $x\in F\cap B_0$ we set
\begin{equation}\label{defhx}
h(x) = \sup\left\{t>0:\,\exists y\in F, \exists\tau,\frac
t3\geq\tau\geq\frac t4, x\in B\bigl(y,\frac \tau3\bigr) \mbox{ and }
(y,\tau)\not\in S_{total}\right\},
\end{equation}
 and
 $$S= \left\{(x,t)\in S_{total}:\,t\geq h(x)\right\}.$$
Notice that if $(x,t)\in S$, then $(x,t')\in S$ for $t'>t$.

Now we consider the following partition of $F\cap B_0$:
\begin{align*}
\ZZ & = \{x\in F\cap B_0:\,h(x)=0\},\\
F_1 & = \left\{x\in F\cap B_0\setminus \ZZ:\, \mbox{$\exists y\in F,
\exists\tau\in\bigl[ \frac{h(x)}5,\frac{h(x)}2\bigr],x\in
B(y,\frac\tau2),\,
\delta(y,\tau)\leq\delta_0$}\right\},\\
F_2 & = \Bigl\{x\in F\cap B_0\setminus (\ZZ\cup F_1):\\
&\qquad \left.\mbox{$\exists y\in
F, \exists\tau\in\!\!\bigl[\frac{h(x)}5,\frac{h(x)}2\bigr],x\in\!
B(y,\frac\tau2),\,
\beta_{1,F}(y,\tau)\geq\ve$}\right\},\\
F_3 & = \Bigl\{x\in F\cap B_0\setminus (\ZZ\cup F_1\cup F_2):\\
&\qquad \left. \mbox{$\exists y\in F, \exists\tau\in\bigl[
\frac{h(x)}5,\frac{h(x)}2\bigr],x\in B(y,\frac\tau2),\, \meas
(D_{y,t},D_0)\geq\tfrac34\alpha$}\right\}.
\end{align*}

\begin{remark}\label{remleg}
It is easy to check that if $x\in F_3$, then for $h(x)\leq t\leq
100h(x)$ we have $\meas(D_{x,h(x)},D_0)\geq\alpha/2$, due to the
fact that $\ve\ll\alpha$. See \cite[Remark~3.3]{Leger}.
\end{remark}

The only difference between the definitions above and the ones in \cite[Subsection 3.1]{Leger} is
that we work with $n$-dimensional densities, $\beta$'s, and planes, while in \cite{Leger}
the dimension is $n=1$.

% ********************************************************************
% ********************************************************************

\subsection{$F_2$ is void}

\begin{lemma}\label{lemf2void}
If $\delta_2$ is small enough in Main Lemma \ref{mlem}, then $F_2$
is void. Moreover, $\beta_{\infty,F}(x,r)\leq \ve^2$ for all $x\in
F$ and $r> 3h(x)$.
\end{lemma}

\begin{proof}
By definition, since $r>3h(x)$, then $(x,r)\in S_{total}$, and then
$\delta_F(x,r)\geq \delta_0$.
 We set
$s:=M_1r_0/\delta_2$. For $y\in F$ with $|x-y|\leq 3r$ and $0<\tau\leq r_0/\delta_2$ we
have
$$|\wt R_\tau \mu(x) - \wt R_\tau \mu(y)| \leq |\wt R_{\tau,s} \mu(x)| + |\wt R_{\tau,s} \mu(y)| +
|\wt R_s \mu(x) - \wt R_s \mu(y)|$$ (notice that $\tau<s$).
 By the smoothness of the kernel of $\wt R_s$ and the assumption (b) in Main Lemma
 \ref{mlem}, it is easy to check that
$$|\wt R_s \mu(x) - \wt R_s \mu(y)| \lesssim \frac{M_1|x-y|}{s}\lesssim \frac{M_1r_0}{s} = \delta_2.$$
Also, by (d) in Main Lemma
 \ref{mlem}, since $s\leq r_0/\delta_2^2$ (for $\delta_2$ small enough), we have
$$|\wt R_{\tau,s} \mu(x)| + |\wt R_{\tau,s} \mu(y)|\leq 2\delta_2.$$
Therefore,
$$|\wt R_\tau \mu(x) - \wt R_\tau \mu(y)| \leq C\delta_2$$
$0<\tau\leq r_0/\delta_2$, and so from Lemma \ref{lemflat}, we
derive $\beta_{\infty,F}(x,r)\leq \ve^2$ for all $x\in F$ and $r\geq2 h(x)$, assuming $\delta_2$ small enough (notice that $\delta_2$ may depend on $\delta_0$).
In particular, this implies that $F_2$ is void.
\end{proof}

Let us remark that we have preferred to maintain the definition of $F_2$ in the preceding subsection in order to keep
the analogy with the construction in \cite{Leger}, although here $F_2$ turns out to be void.

% ********************************************************************
% ********************************************************************

\subsection{The Lipschitz graph and the size of $F_1$}

For $x\in\R^d$ we set
$$d(x) = \inf_{(X,t)\in S}(|X-x| + t),$$
and for $p\in D_0$,
$$D(p) = \inf_{x\in \Pi^{-1}(p)}d(x) = \inf_{(X,t)\in S}(|\Pi(X)-p| + t).$$
Notice that $d$ and $D$ are $1$-Lipschitz functions. Moreover,
$h(x)\geq d(x)$ for $x\in F\cap B_0$, and
$$\ZZ= \{x\in F\cap B_0:d(x)=0\}.$$
Observe also that $d(\cdot)$ is defined on $\R^d$, and not only on $F\cap B_0$. Moreover,
$d(x)\geq r_0$ if $x\notin 2B_0$, since $(X,t)\in S$ implies that $X\in F\cap B_0$.

The construction of the Lipschitz graph $\Gamma$ is basically the
same as the one in \cite{Leger}. The only difference is that in our
case the dimension is $n>1$. So, we have:

\begin{lemma}\label{lemgrlip}
There exists a Lipschitz function $A:\R^n\to\R^{d-n}$ supported on $\Pi(3B_0)$ with $\|\nabla
A\|_\infty\leq C\alpha$ such that if we set $\wt A(p) = (p,A(p))$ for $p\in \R^n$ and
$$\wt F = \{x\in F:\dist(x,\wt A(\Pi(x)))\leq\ve^{1/2}d(x)\},$$
then we have
$$\mu(F\setminus \wt F)\leq C\ve^{1/2}\mu(F).$$
Moreover,
\begin{equation}\label{eqnab22}
|\nabla^2 A(p)|\leq \frac{C\ve}{D(p)}, \qquad p\in\R^n.
\end{equation}
\end{lemma}

See Lemma 3.13 and Proposition 3.8 of \cite{Leger} for the details.

Notice that if $x\notin 2B_0$, then $d(x)>r_0$, and taking into account that $\beta_{\infty,F}(10B_0)\leq\ve^2$, it
turns out that $F\setminus 2B_0\subset \wt F$ (recall also that $F\subset 10 B_0$).

To tell the truth, the Lipschitz graph that is constructed in
\cite{Leger} needs not to be supported on $\Pi(3B_0)$, however  it
is not difficult to show that if one has a Lipschitz graph $A_0$
satisfying the assumptions above except the one on the support, then
one can take $A = A_0\eta$ where $\eta:\R^n\to\R$ is a $\CC^\infty$
function such that $\chi_{\Pi(2B_0)}\eta\leq \chi_{\Pi(3B_0)}$.

\begin{remark}\label{remult}
To prove Main Lemma \ref{mlem} we will show that if parameters $\delta_0$, $\alpha$ and $\ve$ are chosen small enough,
then $\mu(\wt F\cap B_0)\geq \frac{99}{100}\,c_nr_0^n$ (see Lemma~\ref{lemaprox00})
 and the sets $F_1$ and $F_3$ are much smaller that
$\mu(\wt F\cap B_0)$. By the preceding construction and definitions, we have $\wt F\cap B_0\setminus (F_1\cup F_2\cup F_3)\subset \Gamma.$
\end{remark}

Arguing as in \cite[Proposition 3.19]{Leger}, if $\delta_0$ and $\ve$ are small enough, we get

\begin{lemma}
$$\mu(F_1)\leq 10^{-6}\mu(F\cap B_0).$$
\end{lemma}

% ********************************************************************
% ********************************************************************

\subsection{A technical lemma}
The following is a technical result that will be used below.

\begin{lemma}\label{lemtec}
If $x\in F$ and $y\in\R^d$ satisfy $\Pi(x) = \Pi(y)$, then
$$d(x)\lesssim d(y).$$
and so
$$d(x)\approx D(\Pi(x)).$$
\end{lemma}

\begin{proof}
The second assertion is a straightforward consequence of the first
one. So we only have to prove that $d(x)\lesssim d(y)$. Set
$\ell=|x-y|$. We distinguish several cases:
\vv

\noi $\bullet$  If $\ell\leq d(x)/2$, since $d(\cdot)$ is
$1$-Lipschitz, it follows that $|d(x)-d(y)|\leq d(x)/2$, and so
$d(x)\approx d(y)$.

\vv
\noi $\bullet$
Suppose that $d(x)/2<\ell\leq r_0$ and that $d(y)\leq d(x)/8$. By the definition of $d(x)$ it turns
out that there exists some $(X,t)\in S$ such that
$$|X-x| + t\leq 2d(x)<4\ell.$$
Notice that we also have $(X,6\ell)\in S$ (because $\ell\leq r_0$), and thus
$$\beta_{\infty,F}(X,6\ell)\leq \ve^2$$
by Lemma \ref{lemf2void}.
If $D_{X,6\ell}$ stands for the $n$-plane that
minimizes $\beta_{\infty,F}(X,6\ell)$, since $\meas (D_0,D_{X,6\ell})\ll1$ and $\Pi(x)=\Pi(y)$,
\begin{equation}\label{eqfch30}
|x-y|\leq  2\bigl[\dist(x,D_{X,6\ell}) + \dist(y,D_{X,6\ell})\bigr]\leq 12\ve^2\ell +2 \dist(y,D_{X,6\ell}),
\end{equation}
since $x\in F$.

By the definition of $d(y)$, there exists $(Y,u)\in S$ such that
\begin{equation}\label{eqfch3}
|Y-y| + u\leq 2d(y)\leq \frac{d(x)}4\leq\frac{\ell}2.
\end{equation}
 Since
$$|Y-X|\leq |Y-y|+ |y-x|+ |x-X| \leq \ell + \ell +4\ell = 6\ell$$
and $Y\in F$, we also have
$$\dist(Y,D_{X,6\ell})\leq 6\ve^2\ell,$$
by Lemma \ref{lemtec} again. Therefore, by \rf{eqfch3},
$$\dist(y,D_{X,6\ell})\leq  |y-Y| + \dist(Y,D_{X,6\ell}) \leq \frac\ell2 + 12\ve^2\ell.$$
Thus by \rf{eqfch30},
$$|x-y|\leq \frac\ell2 + 24\ve^2\ell <\ell$$
if $\ve$ is small enough, which is a contradiction.

\vv
\noi $\bullet$  Suppose now that $\ell>r_0$. Since $F\subset 10B_0$, $\beta_{\infty,F}(10B_0)\ll1$,
$\Pi(x)=\Pi(y)$, and $|x-y|\geq r_0$, by geometric arguments it easily follows that $\dist(y,F)\gtrsim r_0$.
This implies that $d(y)\gtrsim r_0$ by the definition of $d(y)$, and so $d(y)\gtrsim d(x)$.
\end{proof}

% ********************************************************************
% ********************************************************************
% ********************************************************************
% ********************************************************************

\section{The proof that $F_3$ is small} \label{sec10}

% ********************************************************************
% ********************************************************************

\subsection{The strategy}

For $x\in \R^d$, we set
$$\ell(x) := \frac1{10}\,D(\Pi(x)).$$
Also, for any measure $\sigma$ we denote
$$R_{\ell(\cdot),r_0}^\bot\sigma(x) := \wh R_{\ell(x)}^\bot\sigma(x)
- \wh R_{r_0}^\bot\sigma(x).$$
 For simplicity we have preferred the notation
$R_{\ell(\cdot),r_0}^\bot\sigma(x)$ instead of $\wh
R_{\ell(\cdot),r_0}^\bot\sigma(x)$, although the latter seems more
natural.

 Roughly speaking, the arguments to show that $F_3$ cannot
be too big are the following:
\begin{align*}
F_3 \mbox{ big } & \Rightarrow \|\nabla A\|_2 \mbox{ big }
\Rightarrow \|R^\bot \HH^n_{|\Gamma}\|_{L^2(\Gamma)} \mbox{ big } \\
& \Rightarrow \|R^\bot_{\ell(\cdot),r_0} \HH^n_{|\Gamma\cap
5B_0}\|_{L^2(\Gamma\cap 4B_0)} \mbox{ big } \\ & \Rightarrow
\|R^\bot_{\ell(\cdot),r_0} \mu_{|\wt F}\|_{L^2(\Gamma\cap 4B_0)}
\mbox{ big } \Rightarrow \|R^\bot_{\ell(\cdot),r_0}
\mu\|_{L^2(\mu_{|F})} \mbox{ big, }
\end{align*}
which contradicts the assumptions of Main Lemma \ref{mlem}.

Let us explain some more details. The fact that $\|\nabla A\|_2$ must be big if
$F_3$ is big follows from the definition of $F_3$. Loosely speaking, if $x\in F_3$,
 then the approximating Lipschitz graph has slope $\gtrsim \alpha$ near $x$, by construction.
As a consequence, we should expect
$\|\nabla A\|_2\gtrsim \alpha\mu(F_3)^{1/2}$ (or a similar inequality) to hold.

The implication
$$\|\nabla A\|_2 \mbox{ big }
\Rightarrow \|R^\bot \HH^n_{|\Gamma}\|_{L^2(\Gamma)} \mbox{ big }$$
is a direct consequence of Theorem \ref{teolip}.
Finally, the implications
$$\|R^\bot \HH^n_{|\Gamma}\|_{L^2(\Gamma)} \mbox{ big } \Rightarrow \cdots \Rightarrow
\|R^\bot_{\ell(\cdot),r_0}
\mu\|_{L^2(\mu_{|F})} \mbox{ big }$$
follow, basically, by approximation. For these arguments to work one has to control the ``errors''
in this approximation. In particular, the errors must be smaller than $C\alpha\mu(F_3)^{1/2}$. A key point
here is that these errors depend mostly on the parameter $\ve$ in the definition of $F_2$ and
we have chosen $\ve\ll\alpha$.

% ********************************************************************
% ********************************************************************

\subsection{The implication $F_3$  big  $\Rightarrow$
$\|\nabla A\|_2$ big}\label{sub111}

\begin{lemma}\label{lemf3gran1}
We have
$$\mu(F_3) \leq C\alpha^{-2}\|\nabla A\|_2^2 + C\ve^{1/2}\mu(F).$$
\end{lemma}

\begin{proof}
For a fixed $x\in F_3$, consider the ball $B=B(x,r)$, with $r=2h(x)$
(recall that $h(x)$ was defined in \rf{defhx}). Suppose that
$\mu(B\cap \wt F)\geq \mu(B\cap F)/2$.
  By Lemma \ref{lemli} there are
$n+1$ balls $\Delta_0,\ldots,\Delta_n$ with radius $t/C_{12}$ such that
$\mu(\wt F\cap \Delta_i)\geq C(\delta)^{-1}r^n$ and for all
$(x_0,\ldots,x_n)\in \Delta_0\times\ldots\times \Delta_n$ we have
$${\rm vol}^n((x_0,\ldots,x_n))\geq C^{-1}r^n.$$
By Remark \ref{remleg}, we have $\meas (D_{x,r},D_0)\geq\alpha/2$.
Then, it is easy to check that
$$m_B(|A-m_B(A)|)\geq C^{-1}\alpha \,r,$$
since, for each $i$, $\Delta_i\cap \wt F$ is very close to the graph
of $A$ and also very close to $D_{x,r}$, and moreover
$\ve^{1/2}\ll\alpha$. As a consequence, by Poincaré inequality,
$$m_B(|\nabla A|) \geq C^{-1} \frac{m_B(|A-m_B(A)|)}r \geq C^{-1}
\alpha.$$ Thus, for this ball we have
$$\|\chi_B \nabla A\|_2^2\geq C^{-1} \alpha^2 r^n.$$

Take now a Besicovitch covering of $F_3$ with balls $B_i=B(x_i,r_i)$
as above (i.e.\ $x_i\in F_3$ and $r_i=2h(x_i)$). Denote by $I_1$ the
collection of balls $B_i$ such that $\mu(B_i\cap \wt F)\geq
\mu(B_i\cap F)/2$. We have
\begin{equation}\label{eqdfe1}
\alpha^2\sum_{i\in I_1} \mu(B_i\cap F)\leq C\sum_{i\in
I_1}\|\chi_{B_i} \nabla A\|_2^2 \leq C\|\nabla A\|_2^2.
\end{equation}
 For the
balls $B_i$ in the other collection, that we denote by $I_2$, we
have $\mu(B_i\cap \wt F)< \mu(B_i\cap F)/2$. Thus,
$$\mu(B_i\cap F\setminus \wt F) \geq \frac12\,\mu(B_i\cap F),\qquad i\in I_2.$$
So we get
\begin{equation}\label{eqdfe2}
\sum_{i\in I_2}\mu(B_i\cap F) \leq 2\sum_i \mu(B_i\cap F\setminus
\wt F) \leq C\mu(F\setminus \wt F)\leq C\ve^{1/2}\mu(F).
\end{equation}
The lemma follows from \rf{eqdfe1} and \rf{eqdfe2}.
\end{proof}

% ********************************************************************
% ********************************************************************

\subsection{The implication $\|\nabla A\|_2$ big  $\Rightarrow$ $\|R^\bot
(\HH^n_{\Gamma})\|_{L^2(\Gamma)}$ big} This is a direct consequence
of Corollary \ref{corofac}. Indeed, recall that we showed that
\begin{equation}\label{eqfac2}
\|R^\bot(\HH^n_{\Gamma})\|_{L^2(\Gamma)}\approx \|\nabla A\|_2,
\end{equation}
assuming that $\|\nabla A\|_\infty$ is small enough, which is true
in our construction if $\alpha\ll1$.

% ********************************************************************
% ********************************************************************

\subsection{The implication   $\|R^\bot
(\HH^n_{\Gamma})\|_{L^2(\Gamma)}$ big $\Rightarrow$
$\|R^\bot_{\ell(\cdot),r_0} (\HH^n_{\Gamma\cap
5B_0})\|_{L^2(\Gamma\cap 4B_0)}$ big}

\begin{lemma} \label{lemfff}
$$\|R^\bot (\HH^n_{\Gamma})\|_{L^2(\Gamma)} \leq
\|R^\bot_{\ell(\cdot),r_0} (\HH^n_{\Gamma\cap
5B_0})\|_{L^2(\Gamma\cap 4B_0)} + C\alpha^2\, r_0^{n/2}.$$
\end{lemma}

\begin{proof}
Recall that $\supp (A)\subset 3B_0$. We set
\begin{align*}
\|R^\bot (\HH^n_{\Gamma})\|_{L^2(\Gamma)} & \leq \|\chi _{4B_0}
R^\bot (\HH^n_{\Gamma\cap 5B_0})\|_{L^2(\Gamma)} \\ &\quad + \|\chi
_{4B_0} R^\bot (\HH^n_{\Gamma\setminus 5B_0})\|_{L^2(\Gamma)} +
\|\chi _{\Gamma \setminus 4B_0} R^\bot
(\HH^n_{\Gamma})\|_{L^2(\Gamma)}\\
& = I+ II + III.
\end{align*}

Let us see that the  terms $II$ and $III$ are small. We consider
first $II$. Given $x\in 4B_0$, we have
\begin{align*}
|R^\bot (\HH^n_{\Gamma\setminus 5B_0})(x)| & \leq
\int_{y\in\Gamma:|y-x|\geq r_0}
\frac{|x^\bot-y^\bot|}{|x-y|^{n+1}}\,d\HH^n(y)\\
& = \int_{y\in D_0:|y-x|\geq r_0}
\frac{\dist(x,D_0)}{|x-y|^{n+1}}\,d\HH^n(y) \lesssim
\frac{\dist(x,D_0)}{r_0}.
\end{align*}
If we square  and integrate the last estimate on $4B_0$, we get
$$II^2\lesssim \beta_{2,\Gamma}(2B_0)^2r_0^n \lesssim \ve^2r_0^n.$$

To estimate the term $III$ we take $x\in \Gamma\setminus 4B_0=  D_0\setminus4B_0$ (so
$x^\bot=0$), and we set
\begin{align*}
|R^\bot (\HH^n_{\Gamma})(x)| & \leq
\int_{y\in\Gamma}\frac{\dist(y,D_0)}{|x-y|^{n+1}}\,d\HH^n(y) =
\int_{y\in\Gamma\cap
3B_0}\frac{\dist(y,D_0)}{|x-y|^{n+1}}\,d\HH^n(y)\\
& \approx \frac1{\bigl(r_0 +
|x-x_0|\bigr)^{n+1}}\int_{y\in\Gamma\cap
3B_0}\dist(y,D_0)\,d\HH^n(y) \\ & \lesssim
\beta_{1,\Gamma}(2B_0)\,\frac{r_0^{n+1}}{\bigl(r_0 +
|x-x_0|\bigr)^{n+1}}.
\end{align*}
Squaring and integrating on $D_0\setminus 4B_0$, we obtain
$$III^2\lesssim \beta_{1,\Gamma}(2B_0)^2r_0^n \leq \ve^2r_0^n.$$

%%%%%%%%%%%%%%%%%%%%%%%%%%%%%%%%%%%%%%%%%%%%%%%%%%%%%%%%%%%%%%%%%%%%

To deal with the term $I$, given $x\in \Gamma\cap 4B_0$, we set
\begin{align*}
|R^\bot (\HH^n_{\Gamma\cap 5B_0})(x)| & \leq |R^\bot_{0,\ell(x)}
(\HH^n_{\Gamma\cap 5B_0})(x)|+ |R^\bot_{\ell(x),r_0}
(\HH^n_{\Gamma\cap 5B_0})(x)| \\
& \quad + |\wh R^\bot_{r_0} (\HH^n_{\Gamma\cap 5B_0})(x)|.
\end{align*}
We consider first the term $|\wh R^\bot_{r_0} (\HH^n_{\Gamma\cap
5B_0})(x)|$, for $x\in\Gamma\cap4B_0$:
\begin{align*}
|\wh R^\bot_{r_0} (\HH^n_{\Gamma\cap 5B_0})(x)| & \leq
\int_{y\in\Gamma\cap5B_0:|y-x|> r_0/2}
\frac{|x^\bot-y^\bot|}{|x-y|^{n+1}}\,d\HH^n_{\Gamma}(y)\\
&\lesssim \int_{y\in\Gamma\cap5B_0}
\frac{\dist(x,D_0)+\dist(y,D_0)}{r_0^{n+1}}\,d\HH^n_{\Gamma}(y)
\lesssim \beta_{\infty,\Gamma}(2B_0).
\end{align*}
So we get
\begin{equation}\label{eqdeix}
\|\chi_{4B_0}\wh R^\bot_{r_0} (\HH^n_{\Gamma\cap
5B_0})\|_{L^2(\Gamma)}^2\lesssim
\beta_{\infty,\Gamma}(2B_0)^2r_0^n\lesssim \ve^2r_0^n.
\end{equation}

To estimate $|R^\bot_{0,\ell(x)} (\HH^n_{\Gamma\cap 5B_0})(x)|$ we
will use the smoothness of $\Gamma$ on the stopping cubes. That is,
we will use the estimate \rf{eqnab22}. Notice first that
$$R^\bot_{0,\ell(x)} (\HH^n_{\Gamma\cap 5B_0})(x) = R^\bot_{0,\ell(x)}
(\HH^n_{\Gamma})(x)$$ for $x\in 4B_0$, since $\ell(x)< r_0$. So if
we set $x=\wt A(p)$, $y=\wt A(q)$, with $p,q\in\R^n$, we have
\begin{align}\label{eqat1}
R^\bot_{0,\ell(x)}\HH^n_{\Gamma\cap 5B_0}(x) & =
\!\int\biggl(1-\psi\Bigl(\frac{\wt A(p) - \wt
A(q)}{D(p)/10}\Bigr)\biggr)\,
 \frac{A(p)-A(q)}{\bigl|\wt A(p)- \wt A(q)\bigr|^{n+1}}
\,J(\wt A)(q)dq,
\end{align}
where $dq$ stands for the $n$-dimensional Lebesgue measure. We
denote by $S(x)$ the integral on the right hand side of \rf{eqat1},
and we set
\begin{align*}
S(x) & = \int\biggl(1-\psi\Bigl(\frac{p - q}{D(p)/10}\Bigr)\biggr)\,
 \frac{A(p)-A(q)}{\bigl|\wt A(p)- \wt A(q)\bigr|^{n+1}} \,dq \\
& \quad \!\!\! + \int\biggl(\psi\Bigl(\frac{p -
q}{D(p)/10}\Bigr)-\psi\Bigl(\frac{\wt A(p) - \wt
A(q)}{D(p)/10}\Bigr)\biggr)\,
 \frac{A(p)-A(q)}{\bigl|\wt A(p)- \wt A(q)\bigr|^{n+1}} \,dq
 \\
&\quad \!\!\!+ \int\biggl(1-\psi\Bigl(\frac{\wt A(p) - \wt
A(q)}{D(p)/10}\Bigr)\biggr)\,
 \frac{A(p)-A(q)}{\bigl|\wt A(p)- \wt A(q)\bigr|^{n+1}} \,\bigl(J(\wt A)(q) - 1\bigr)dq\\
&= S_1(x) + S_2(x) + S_3(x).
\end{align*}
Recall that by Remark \ref{remjac} we have
$$\|J(\wt A) - 1\|_2\lesssim \|\nabla A\|_\infty\|\nabla A\|_2.$$
So, by the $L^2$ boundedness of Riesz transforms on Lipschitz graphs
we get
\begin{equation}\label{eqs2f}
\|S_3\|_2\lesssim \|\nabla A\|_\infty\|\nabla A\|_2.
\end{equation}

To deal with $S_2$ notice that
\begin{equation}\label{eqmean1}
\bigl| |\wt A(p) -\wt A(q)| - |p-q|\bigr| \leq |A(p)-A(q)| \leq C\alpha|p-q| \leq \frac12|p-q|,
\end{equation}
and so
$$\frac12 |p-q| \leq |\wt A(p) -\wt A(q)|\leq 2|p-q|.$$
Since $\psi(z) =0$ if $|z|\leq1/2$ and $\psi(z) =1$ if $|z|\geq1$,
we deduce that
$$\psi\Bigl(\frac{p - q}{D(p)/10}\Bigr)-
\psi\Bigl(\frac{\wt A(p) - \wt A(q)}{D(p)/10}\Bigr) = 0$$ if
$|p-q|\leq D(p)/40$ or $|p-q|\geq D(p)/5$. Moreover, from the mean
value theorem and \rf{eqmean1},
$$\biggl| \psi\Bigl(\frac{p - q}{D(p)/10}\Bigr)-
\psi\Bigl(\frac{\wt A(p) - \wt A(q)}{D(p)/10}\Bigr)\biggr| \leq
\frac{C\alpha|p-q|}{D(p)}.$$ Thus,
\begin{align*}
|S_2(x)| & \lesssim \int_{D(p)/40\leq |p-q|\leq D(p)/5}
\frac{\alpha|p-q|}{D(p)} \,\frac{|A(p)-A(q)|}{|p-q|^{n+1}}\,
dq\\
& \lesssim \frac{\alpha^2}{D(p)}
\int_{D(p)/40\leq |p-q|\leq D(p)/5}  \frac{1}{|p-q|^{n-1}}\,dq\lesssim\alpha^2.
\end{align*}
Therefore,
\begin{equation}\label{eqwl4}
\|S_2\|_2\lesssim\alpha^2\,r_0^{n/2}.
\end{equation}

%%%%%%%%%%%%%%%%%%%%%%%%%%%%%%%%%%%%%%%%%%%%%%%%%%%%%%%%%%%%%%

We are left with the term $S_1(x)$. By Taylor's formula, we have
\begin{multline*}
\frac{A(p)-A(q)}{\bigl(|p-q|^2+ |A(p)-A(q)|^2\bigr)^{(n+1)/2}} \\ =
\sum_{k=0}^\infty (-1)^k\frac{(n+2k-1)!!}{2^k k!}\cdot
\frac{|A(p)-A(q)|^{2k}}{|p-q|^{n+2k+1}}\,(A(p)-A(q)).
\end{multline*}
The series is uniformly convergent since $|A(p)-A(q)|/|p-q|\leq
C\alpha\ll1$.
Notice that the integrand in $S_1$ vanishes if $|p-q|> D(p)/10$.
On the other hand, by Taylor's formula and \rf{eqnab22} we also have
$$A(p) - A(q) = \nabla A(p)(p-q) + E(p,q),$$
with
\begin{align}\label{eqe11}
|E(p,q)|&\leq C\sup_{z\in B(p,D(p)/10)}|\nabla^2 A(z)|\,|p-q|^2 \\
& \leq \ve|p-q|^2
\sup_{z\in B(p,D(p)/10)}\frac{1}{D(z)} \lesssim
\frac{\ve|p-q|^2}{D(p)}, \nonumber
\end{align}
since $D(\cdot)$ is $1$-Lipschitz.
 Then it turns out that
$$(A(p)- A(q))\,|A(p)-A(q)|^{2k} = \nabla A(p)(p-q)\,|\nabla A(p)(p-q)|^{2k} + E_k(p,q),$$
with\footnote{For this estimate we take into account that $\|\nabla
A\|_\infty + \ve \leq 1/4$ and we use the fact that $(a+b)^m = a^m +
c$, with $|c|\leq 2^k|b|\max(|a|,|b|)^{m-1}$.}
$$|E_k(p,q)| \leq C\ve 2^{-k}\,\frac{|p-q|^{2k+2}}{D(p)}.$$
We have
\begin{align*}
S_1(x) & = \!\sum_{k=0}^\infty (-1)^k\frac{(n+2k-1)!!}{2^k
k!} \\
& \quad \;\times \int \biggl(1-\psi\Bigl(\frac{p -
q}{D(p)/10}\Bigr)\biggr)\,\frac{\nabla A(p)(p-q)\,|\nabla
A(p)(p-q)|^{2k}}{|p-q|^{n+2k+1}} \,dq\\
&\quad + \sum_{k=0}^\infty (-1)^k\frac{(n+2k-1)!!}{2^k k!}
\int_{|p-q|\leq D(p)/10}\frac{E_k(p,q)}{|p-q|^{n+2k+1}}\,dq.
\end{align*}
Notice that the first sum on the right side above vanishes because
each integral in the sum equals zero by the antisymmetry of the
integrand. We obtain
\begin{align}\label{eqs1f}
|S_1(x)|& \leq\sum_{k=0}^\infty \frac{(n+2k-1)!!}{2^k k!}
\int_{|p-q|\leq D(p)/10} \frac{|E_k(p,q)|}{|p-q|^{n+2k+1}}\,dq\
 \\
& \lesssim  \sum_{k=0}^\infty \frac{(n+2k-1)!!}{4^k k!}
\int_{|p-q|\leq
D(p)/10} \frac{\ve}{D(p)|p-q|^{n-1}}\,dq \nonumber\\
& \approx \sum_{k=0}^\infty \frac{(n+2k-1)!!}{4^k k!}\,\ve \approx
\ve.\nonumber
\end{align}

From \rf{eqat1}, \rf{eqs2f}, \rf{eqwl4}, and \rf{eqs1f} we deduce that
$$\|\chi_{4B_0}R^\bot_{0,\ell(x)} (\HH^n_{\Gamma\cap 5B_0})\|_{L^2(\Gamma)}
\lesssim \|\nabla A\|_\infty\|\nabla A\|_2+\alpha^2
r_0^{n/2}\lesssim \alpha^2r_0^{n/2}$$ because $\alpha^2\ll\ve$ and
$\|\nabla A\|_2\lesssim\alpha r_0^{n/2}$, since $A$ is supported on
$\Pi(3B_0)$. Therefore, by \rf{eqdeix},
$$I \leq  C(\ve + \alpha^2) r_0^{n/2} \lesssim \alpha^2 r_0^{n/2}.$$
The lemma follows from the preceding estimate and the ones obtained
above for the terms $II$ and $III$.
\end{proof}

% ********************************************************************
% ********************************************************************

\subsection{The implication
$\|R^\bot_{\ell(\cdot),r_0} (\HH^n_{\Gamma\cap
5B_0})\|_{L^2(\Gamma\cap 4B_0)}$ big  $\Rightarrow$
$\|R^\bot_{\ell(\cdot),r_0} \mu_{\wt F}\|_{L^2(\Gamma)}$ big}

This implication is one of the most delicate steps of the proof that $F_3$ is a small set. Let
$\vphi:\R^n\to\R$ be a smooth radially non increasing function with
$\|\vphi\|_1=1$ such that $\supp(\vphi)\subset B_n(0,1)$ and $\vphi$
equals $1$ on $B_n(0,c_0)$ for some $0<c_0< 1$ which may depend on
$n$. As usual, for $t>0$ we denote,
$$\vphi_t(x)=\frac1{t^n}\,\vphi\Bigl(\frac xt\Bigr),\qquad x\in\R^n.$$
Then we consider the function $g:\R^n\to \R$ given by
$$g(x) = \vphi_{\ve^{1/4} D(x)} * \Pi_\#(\mu_{|\wt F})(x).$$
We will show below that $g(x)\, dx$ is very close to the measure
$dx$ on $B(x_0,6r_0)$, in a sense.

First we need the following preliminary result:

\begin{lemma}\label{lemdifphi}
For all $x,y\in\R^n$,
$$|\vphi_{\ve^{1/4}
D(x)}(x-y) - \vphi_{\ve^{1/4} D(y)}(x-y)| \lesssim
\frac{\ve^{1/4}}{(\ve^{1/4}D(y))^n}\,\chi_{B(0,C\ve^{1/4}D(y))}(x-y).$$
\end{lemma}

\begin{proof}
For any $z\in\R^n$
and $s,t>0$ with $s\approx t$,
\begin{align*}
|\vphi_s(z) - \vphi_t(z)| & \leq \Bigl| \frac1{s^n} -
\frac1{t^n}\Bigr|\, \vphi\Bigl(\frac zs\Bigr) +
\frac1{t^n}\,\Bigl|\vphi\Bigl(\frac zs\Bigr) - \vphi\Bigl(\frac
zt\Bigr)\Bigr| \\
& \leq \frac{C|s-t|}{s^{n+1}} \,\vphi\Bigl(\frac zs\Bigr) + \frac
C{t^n}\,\Bigl|\frac zs  - \frac zt\Bigr|  \leq
\frac{C|s-t|}{s^{n+1}},
\end{align*}
since we may assume that $|z|\lesssim s$. As a consequence,
$$|\vphi_s(z) - \vphi_t(z)| \leq
\frac{C|s-t|}{s^{n+1}}\,\chi_{B(0,Cs)}(z).$$

We set $s=\ve^{1/4}D(y)$ and $t=\ve^{1/4}D(x)$. Notice that
$\vphi_{\ve^{1/4} D(x)}(x-y)\neq 0$ implies that $|x-y|\leq
\ve^{1/4}D(x)$, and then it turns out that $D(x)\approx D(y)$. Of
course, the same happens if $\vphi_{\ve^{1/4} D(y)}(x-y)\neq 0$. In
both cases
 we have
$$\frac{|s-t|}{s} = \frac{|D(x)-D(y)|}{D(y)} \leq \frac{|x-y|}{D(y)}
\lesssim \ve^{1/4}.$$ Therefore,
$$|\vphi_{\ve^{1/4}
D(x)}(x-y) - \vphi_{\ve^{1/4} D(y)}(x-y)| \lesssim
\frac{\ve^{1/4}}{(\ve^{1/4}D(y))^n}\,\chi_{B(0,C\ve^{1/4}D(y))}(x-y).$$
\end{proof}

\begin{lemma}\label{lemcreixfi}
Let $\nu$ be a Borel measure on $\R^n$ such that
$$\nu(B(x,r))\leq r^n \qquad \mbox{ for all $x\in\supp(\nu)$ and $r\geq \eta r_0$.}$$
For any $\delta>0$, if $\eta>0$ is small enough, we have
$$\nu(B(x,r))\leq (1+\delta) r^n \qquad\mbox{for all $x\in\R^n$ and $r\geq r_0$.}$$
Moreover, $\eta$ only depends on $\delta$.
\end{lemma}

\begin{proof}
Given a ball $B(x,r)$ with $r\geq r_0$ we can consider a family of disjoint balls $B_i$ contained in $B(x,r)$
centered at points in $\supp(\nu)$ with radii $r_i\geq \eta^{1/2} r$,  such that
$$\LL^n\Bigl(B(x,r)\setminus \bigcup_i B_i\Bigr) \leq \frac\delta2 r^n,$$
assuming that $\eta$ is small enough. Then,
\begin{equation}\label{eqfh33}
\nu(B(x,r))= \sum_{i} \nu(B_i) + \nu\Bigl(B(x,r)\setminus \bigcup_i B_i\Bigr) \leq
r^n+ \nu\Bigl(B(x,r)\setminus \bigcup_i B_i\Bigr).
\end{equation}
We consider now a Besicovitch covering of $\supp(\nu)\cap  B(x,r)\setminus \bigcup_i B_i$ with balls
$B_j'$ centered at points in $\supp(\nu)\cap  B(x,r)\setminus \bigcup_i B_i$ with radii $r_j'= \eta r$ for all $j$.
Then we have
\begin{equation}\label{eqfh34}
\nu\Bigl(B(x,r)\setminus \bigcup_i B_i\Bigr) \leq  \sum_j \nu(B_j')\leq \sum_j (r_j')^n \leq
C\LL^n\Bigl(U_{\eta r}\Bigl(B(x,r)\setminus \bigcup_i B_i\Bigr)\Bigr),
\end{equation}
where $U_{\eta r}(A)$ denotes the $\eta r$-neighborhood of $A$.
We have
$$U_{\eta r}\Bigl(B(x,r)\setminus \bigcup_i B_i\Bigr) \subset
\Bigl(B(x,r)\setminus \bigcup_i B_i\Bigr) \cup U_{\eta r}(\partial B(x,r)) \cup \bigcup_i U_{\eta r}(\partial B_i),$$
and so
$$\LL^n\Bigl(U_{\eta r}\Bigl(B(x,r)\setminus \bigcup_i B_i\Bigr)\Bigr) \leq \frac\delta2 r^n + C\eta r^n + C\sum_i r_i^{n-1}
(\eta r).$$
Since $\eta r\leq \eta^{1/2} r_i$ for all $i$, we get
\begin{align*}
\LL^n\Bigl(U_{\eta r}\Bigl(B(x,r)\setminus \bigcup_i B_i\Bigr)\Bigr)
 & \leq \Bigl(\frac\delta2+ C\eta\Bigr) r^n + C \eta^{1/2}\sum_i r_i^n\\
& \leq  \Bigl(\frac\delta2+ C\eta + C\eta^{1/2}\Bigr) r^n \leq \Bigl(\frac\delta2+ C\eta^{1/2}\Bigr) r^n.
\end{align*}
From \rf{eqfh33} and \rf{eqfh33} we infer that
$$\nu(B(x,r))\leq \Bigl(1+ \frac\delta2+ C\eta^{1/2}\Bigr) r^n,$$
and the lemma follows if $\eta$ is small enough.
\end{proof}

In next lemma we show that $g$ is very close to the function identically $1$ on $8B_0$. We also prove that
$\mu(\wt F\cap B_0)$ is big, which was already mentioned in Remark \ref{remult}.

\begin{lemma}\label{lemaprox00}
If $\ve$ has been chosen small enough and $\delta_1\leq \alpha^2$
(where $\delta_1$ is the constant from (a) and (b) in Main Lemma
\ref{mlem}), then we have
\begin{equation}\label{eqmy0}
\Pi_\#(\mu_{|\wt F})(B(p,r))\leq c_n(1+c\alpha^2)r\quad \mbox{ for
all $p\in\R^n$ and $r\geq\ve^{1/2}D(p)$},
\end{equation}
\begin{equation}\label{eqgmenys}
0\leq g(p)\leq 1+C_{20}\alpha^2 \quad \mbox{ for all $p\in\R^n$,}
\end{equation}
\begin{equation}\label{eqgnorm}
\|\chi_{8B_0}(g-1)\|_1\leq C\alpha ^2\,r_0^n,
\end{equation}
and
\begin{equation}\label{eqgnorm2}
\|\chi_{8B_0}(g-1)\|_2\leq C\alpha\,r_0^{n/2}.
\end{equation}
Also,
\begin{equation}\label{eqcla77}
\mu(\wt F \cap B_0)\geq \frac{99}{100}\,c_nr_0^n.
\end{equation}

\end{lemma}

\begin{proof}
First we will show \rf{eqmy0}. Since for all $x\in
\Pi^{-1}(B_n(p,t))\cap \wt F$ (recall that $B_n(p,r)$ is an
$n$-dimensional ball in $\R^n$), we have $\beta_{\infty,F}(x,t)\leq
C\ve$ for $t\geq D(p)$, we infer that there exists some $n$-plane
$L$ such that
$$\Pi^{-1}(B_n(p,r))\cap \wt F \subset U_{C\ve^{1/2} r}( L) \quad \mbox{ if $r\geq \ve^{1/2} D(p)$.}$$
Further, by construction, the $n$-plane $L$ satisfies $\meas
(L,\R^n)\leq C\alpha$. All together, this implies that there exists
some ball $B(z,R)\subset \R^d$, with
$$R\leq (1+C\sin(\alpha)^2)^{1/2} r+ C \ve^{1/2}r \leq (1+C\alpha^2+C\ve^{1/2})r,$$
such that
$$\Pi^{-1}(B_n(p,r))\cap \wt F \subset B(z,R).$$
If $p\in\Pi(\wt F)$, then we may take $z\in \wt F$, and so by the
assumption (b) in the Main Lemma \ref{mlem},
\begin{align} \label{eqsupf}
\Pi_\#\mu_{|\wt F}(B_n(p,r))& \leq \mu(B(z,R))\leq c_n (1+\delta_1)
(1+C\alpha^2+C\ve^{1/2})^n r^n \nonumber\\ & \leq c_n (1+ \delta_1+
C_{21}\alpha^2) r^n
\end{align}
for $r\geq \ve^{1/2}D(p)$ (recall that $\ve^{1/2}\ll\alpha^2$).

Consider now the case $p\not\in\Pi(\wt F)$. Suppose that
$\ve^{1/4}D(p)\leq r \leq D(p)$ and let $\nu = \Pi_\#\mu_{|\wt F\cap
B(p,D(p)/10)}$. By \rf{eqsupf},
$$\nu(B(z,r)) \leq c_n(1+ \delta_1+ C_{21}\alpha^2) r^n$$
for all $z\in\supp(\nu)$ and $r\geq \ve^{1/2}D(z)\approx
\ve^{1/2}D(p)$. From Lemma \ref{lemcreixfi} we deduce that
$$\nu(B(p,r)) \leq c_n(1+ \delta_1+ 2C_{21}\alpha^2) r^n$$
if $r\geq c_0\ve^{1/4}D(p)$ and $\ve$ is small enough (recall that
$c_0$ was defined at the beginning of the current subsection).

To prove \rf{eqgmenys} for a given $p\in\R^n$, let $\psi:\R\to\R$ be such that $\psi(|q|)=
\vphi_{\ve^{1/4} D(p)} (q)$ and denote $\sigma= \Pi_\#(\mu_{|\wt F})$. We have
\begin{align}\label{eqgp1}
g(p) & = \int \psi(|p-q|)\,d\sigma(p) =
-\int_0^\infty \int_{|p-q|}^\infty\psi'(r)\,dr\,d\sigma(p) \nonumber \\
& = -\int_0^\infty \Pi_\#\mu_{|\wt F}(B_n(p,r)) \,\psi'(r)\,dr.
\end{align}
Notice that
$$\sigma(B_n(p,r)) = \Pi_\#\mu_{|\wt F}(B_n(p,r))=\mu(\Pi^{-1}(B_n(p,r))\cap \wt F).$$
Moreover, $\supp (\psi')\subset [c_0 \ve^{1/4} D(p),\ve^{1/4}
D(p)]$, and so
$$g(p) = - \int_{c_0\ve^{1/4} D(p)}^{\ve^{1/4} D(p)} \mu(\Pi^{-1}(B_n(p,r))\cap \wt F)\,\psi'(r)\,dr.$$
Thus,
$$|g(p)|\leq c_n(1+ \delta_1+ 2C_{21}\alpha^2) \int_{c_0\ve^{1/4}D(p)}^{\ve^{1/4}D(p)} r^n \,|\psi'(r)|\,dr =
1+ \delta_1+ 2C_{21}\alpha^2,$$
and \rf{eqgmenys} follows.

Now we turn our attention to \rf{eqgnorm}. First we will show that
\begin{equation}\label{clam67}
\int_{B_n(x_0,8r_0)} g(p) dp \geq (1-C\ve^{1/4})\LL^n(8B_0\cap
\R^n).
\end{equation}
Since $D(p)\leq 9r_0$ for all $p\in\Pi(8B_0)$, we have
\begin{align}\label{eqkk23}
\int_{B_n(x_0,(8+9\ve^{1/4})r_0)} g(p)\,dp & =
\int_{B_n(x_0,(8+9\ve^{1/4})r_0)}
 \vphi_{\ve^{1/4} D(p)} * \sigma(p)\,dp \\
& = \int_{p\in B_n(x_0,(8+9\ve^{1/4})r_0)}\int \vphi_{\ve^{1/4}
D(p)}(p-q)\,d\sigma(q)\,dp \nonumber\\
& \geq \int_{q\in B_n(x_0,8r_0)}\int \vphi_{\ve^{1/4}
D(p)}(p-q)\,dp\,d\sigma(q).\nonumber
\end{align}
  Recall now that by Lemma \ref{lemdifphi},
$$|\vphi_{\ve^{1/4}
D(p)}(p-q) - \vphi_{\ve^{1/4} D(q)}(p-q)| \lesssim
\frac{\ve^{1/4}}{(\ve^{1/4}D(q))^n}\,\chi_{B(q,C\ve^{1/4}D(q))}(p).$$
From this inequality and \rf{eqkk23} we get
\begin{align}\label{eqdcv}
\int_{B_n(x_0,(8+9\ve^{1/4})r_0)} g(p)\,dp & \geq \int_{q\in
B_n(x_0,8r_0)}\int \vphi_{\ve^{1/4} D(q)}(p-q)\,dp\,d\sigma(q) \\
&\quad - \int_{q\in B_n(x_0,8r_0)} \frac{\ve^{1/4} \LL^n(
B(q,C\ve^{1/4}D(q)))}{(\ve^{1/4}D(q))^n} \,d\sigma(q) \nonumber\\
&= (1-C\ve^{1/4})\,\sigma(B_n(x_0,8r_0)) \nonumber
\\ & \geq
(1-C\ve^{1/4})\LL^n(B_n(x_0,8r_0)).\nonumber
\end{align}
The last inequality follows from the assumption (a) of Main Lemma \ref{mlem} and
the fact that $\mu(F\setminus \wt F)\lesssim\ve^{1/2}\mu(F)$.
Inequality \rf{clam67} is a consequence of \rf{eqdcv} and the estimate
$\|g\|_\infty\leq 2$ (by\rf{eqgmenys}).

The estimate \rf{eqgnorm} is a direct consequence of \rf{eqgmenys}
and \rf{clam67}:
\begin{align*}\int_{\Pi(8B_0)}
\!|(1+C_{20}\alpha^2) - g(p)|\,dp & =
\int_{\Pi(8B_0)} \bigl((1+C_{20}\alpha^2) - g(p)\bigr)\,dp \\
& = (1+C_{20}\alpha^2)\LL^n(\Pi(8B_0)) - \int_{\Pi(8B_0)}\!\!\!\! g(p)\,dp\\
& \leq (C_{20}\alpha^2 + C\ve^{1/4})\LL^n(\Pi(8B_0)).
\end{align*}
Thus,
$$\int_{\Pi(8B_0)}
|1 - g(p)|\,dp \leq (2C_{20}\alpha^2 + C\ve^{1/4})\LL^n(\Pi(8B_0)),$$
and so we get \rf{eqgnorm} if $\ve$ is small enough.

On the other hand, \rf{eqgnorm2} is a direct consequence of \rf{eqgnorm}:
$$\int_{\Pi(8B_0)}
|1 - g(p)|^2\,dp \leq (1+\|g\|_\infty) \int_{\Pi(8B_0)} |1 -
g(p)|\,dp \leq C\alpha^2\,r_0^n.$$

Finally we deal with \rf{eqcla77}:
if we argue as in \rf{eqkk23} and \rf{eqdcv},
with $B_n(x_0,8r_0) \setminus B_n(x_0,\frac{999}{1000}r_0)$ instead of $B_n(x_0,8r_0)$, we get
\begin{align*}
\sigma(\Pi(8B_0) \setminus \Pi(\tfrac{999}{1000}B_0)) & \leq
\int_{B_n(x_0,(8+9\ve^{1/4})r_0)\setminus B_n(x_0,(1-9\ve^{1/4})\frac{999}{1000}r_0)} g(p)\,dp + C\ve^{1/4}r_0^n \\
& \leq c_n(1+C\alpha^2) (8^n-\tfrac{999}{1000}) r_0^n + C\ve^{1/4}r_0^n.
\end{align*}
Since $\mu(\wt F\cap B_0)\geq \sigma(\Pi(\frac{999}{1000}B_0))$ if $\beta_{\infty,F}(B_0)$ is small enough,
we have
\begin{align*}
\mu(\wt F\cap B_0) & \geq   \sigma(\Pi(8B_0)) - \sigma(\Pi(8B_0)\setminus\Pi(\tfrac{999}{1000}B_0))\\
& \geq c_n8^nr_0^n - C\ve^{1/2}r_0^n - c_n(1+C\alpha^2) (8^n-\tfrac{999}{1000}) r_0^n - C\ve^{1/4}r_0^n\\
& \geq \frac{99}{100}\,c_nr_0^n,
\end{align*}
if $\alpha$ and $\ve$ are small enough.
\end{proof}

%%%%%%%%%%%%%%%%%%%%%%%%%%%%%%%%%%%%%%%%%%%%%%%%%%%%%%%%%%%%%%%%%%%%%%%%%%%%
Recall that
 $\Pi$ stands for the orthogonal projection of $\R^d$ onto $D_0\equiv \R^n$, and  $\sigma =\Pi_\#\mu_{|\wt F}$.
We also denote by
$P$ the projection from $\R^d$ onto $\Gamma$ which is
orthogonal to $D_0\equiv\R^n$. Moreover, for $x\in\Gamma$ we set
$$h(x) = \frac{g(\Pi(x))}{J\wt A(\Pi(x))},$$
so that $h(x)\,d\HH^n_{|\Gamma}(x)$ is the image measure of $g(x)\,dx$ by $P$.

\begin{lemma}\label{lemauxxx}
If $f:\R^d\to\R$ is a function with $\supp(f)\subset 5B_0$, then we have
\begin{align}\label{eqaux88}
\biggl|\int_{P(5B_0)}& f(x)\,h(x)\,d\HH^n(x) - \int_{5B_0\cap \wt F}f(x) \,d\mu(x)\biggr| \\ & \leq
\iint_{\begin{subarray}{l}
p\in\Pi(6B_0) \\|p-q|\leq \ve^{1/4} D(q)\end{subarray}}\frac{C}{\bigl(\ve^{1/4} D(q)\bigr)^n}\, \bigl|f(\wt A(p)) - f(\wt A(q))\bigr|\,d\sigma(q)dp\nonumber\\
&\quad + \biggl|
\int_{p\in\Pi(6B_0)} f(\wt A(p))\,b(p) \,dp\biggr| +\int_{5B_0\cap \wt F} \bigl|f(P(x)) - f(x)\bigr|\,d\mu(x),
\nonumber
\end{align}
where $b(p)$ is some function satisfying $\|b\|_\infty\lesssim \ve^{1/4}$.
\end{lemma}

\begin{proof}
We have
\begin{multline*}
\int_{P(5B_0)} f\,h\,d\HH^n - \int_{5B_0\cap \wt F}f \,d\mu \\
\begin{split}
& =
\biggl(\int_{\Pi(5B_0)} f(\wt A(p))\,g(p)\,dp - \int_{\Pi(5B_0)} f(\wt A(p))\,d\Pi_\#\mu_{|\wt F}(p)\biggr)\\
&\quad +
\biggl(\int_{5B_0\cap \wt F} f(P(x))\,d\mu(x) -
\int_{5B_0\cap \wt F}f(x) \,d\mu(x)\biggr) =: S+ T.
\end{split}
\end{multline*}
For this identity we took into account that
$$ \int_{\Pi(5B_0)}\!\!\! f(\wt A(p))\,d\Pi_\#\mu_{|\wt F}(p) = \int_{P(5B_0)} \!\!\!f(x)\,dP_\#\mu_{|\wt F}(x) =
\int_{5B_0\cap \wt F}f(P(x))\,d\mu(x).$$

To estimate the term $S$ we recall that
$$g(p) = \vphi_{\ve^{1/4} D(p)} * \Pi_\#(\mu_{|\wt F})(p) = \vphi_{\ve^{1/4} D(p)} * \sigma(p),$$
and so
\begin{align*}
S & =
\iint_{p\in\Pi(5B_0)} f(\wt A(p))\,\vphi_{\ve^{1/4} D(p)}(p-q)\,d\sigma(q)dp -
\int_{q\in \Pi(6B_0)} f(\wt A(q))\,d\sigma(q) \\
& = \iint_{p\in\Pi(6B_0)} \bigl[f(\wt A(p)) - f(\wt A(q))\bigr]\,\vphi_{\ve^{1/4} D(q)}(p-q)\,d\sigma(q)dp\\
& \quad + \iint_{p\in\Pi(6B_0)} f(\wt A(p))\,\bigl[\vphi_{\ve^{1/4} D(p)}(p-q)-\vphi_{\ve^{1/4} D(q)}(p-q)\bigr] \,d\sigma(q)dp \\
& =: S_1+ S_2,
\end{align*}
since
$$\int_{p\in\Pi(6B_0)} \vphi_{\ve^{1/4} D(q)}(p-q)\,dp = 1\quad\mbox{ for $q\in\supp(f\circ\wt A)$.}$$

Clearly, we have
$$|S_1| \lesssim
\iint_{\begin{subarray}{l}
p\in\Pi(6B_0) \\|p-q|\leq \ve^{1/4} D(q)\end{subarray}} \frac1{\bigl(\ve^{1/4} D(q)\bigr)^n}\bigl|f(\wt A(p))
 - f(\wt A(q))\bigr|\,d\sigma(q)dp.$$
To deal with $S_2$ we denote
$$b(p) = \int \bigl[\vphi_{\ve^{1/4} D(p)}(p-q)-\vphi_{\ve^{1/4} D(q)}(p-q)\bigr] \,d\sigma(q).$$
By Lemma \ref{lemdifphi}
$$|b(p)| \lesssim
\frac{\ve^{1/4}\sigma(B(p,C\ve^{1/4}D(q)))}{(\ve^{1/4}D(q))^n}\lesssim \ve^{1/4}.$$

Concerning the term $T$, we have
$$|T|\leq \int_{5B_0\cap \wt F} \bigl|f(P(x)) - f(x)\bigr|\,d\mu(x).$$
\end{proof}

\begin{lemma}\label{lemaprox1}
$$\|R^\bot_{\ell(\cdot),r_0} (\mu_{|\wt F\cap 5B_0}) - R^\bot_{\ell(\cdot),r_0} (h\,\HH^n_{|\Gamma\cap
5B_0})\|_{L^2(\Gamma\cap 4B_0)}\lesssim \ve^{1/4} r_0^{n/2}.$$
\end{lemma}

\begin{proof}
For any $x\in\Gamma\cap 4B_0$ we have
\begin{multline*}
B(x):= R^\bot_{\ell(\cdot),r_0} (\mu_{|\wt F\cap 5B_0})(x) - R^\bot_{\ell(\cdot),r_0} (h\,\HH^n_{|\Gamma\cap
5B_0})(x)
\\ = \int_{y\in
\wt F\cap 5B_0} K^\bot_{\ell(x),r_0} (x-y) d\mu(y) -
\int_{y\in P(5B_0)}
K^\bot_{\ell(x),r_0} (x-y)\,h(y)\,d\HH^n(y).
\end{multline*}
To estimate $B(x)$ we apply Lemma \ref{lemauxxx} with $f(y) = K^\bot_{\ell(x),r_0} (x-y)$, for each fixed $x\in
\Gamma\cap 4B_0$.
To simplify notation, we set $F(x,p) :=K^\bot_{\ell(x),r_0} (x-\wt A(p))$.
The first term on the right side of \rf{eqaux88} for this choice of $f$ is
$$
B_1(x) := \iint_{\begin{subarray}{l}
p\in\Pi(6B_0) \\|p-q|\leq \ve^{1/4} D(q)\end{subarray}} \!\!\frac{C}{\bigl(\ve^{1/4} D(q)\bigr)^n}
\bigl|F(x,p)-F(x,q)\bigr|
\,d\sigma(q)dp.$$
For $p,q$ satisfying $|p-q|\leq \ve^{1/4}D(p)\lesssim \ve^{1/4}D(q)$ we have
$$\bigl|F(x,p)-F(x,q)\bigr| \lesssim  \frac{\ve^{1/4}
D(q)}{\bigl(|x-\wt A(q))| + \ell(x)\bigr)^{n+1}}.$$
 Moreover,
$$D(q)\leq 10\ell(x) + |D(q) - 10\ell(x)|  \leq 10\ell(x) + 10|\Pi(x)-q|\lesssim\ell(x) + |x-\wt A(q)|,$$
since $D(\Pi(x))=10\ell(x)$ and $D$ is $1$-Lipschitz. So we get
$$\bigl|F(x,p)-F(x,q)\bigr| \lesssim  \frac{\ve^{1/4}
D(q)}{\bigl(D(q) + |x-\wt A(q)|\bigr)^{n+1}}.$$
 Thus,
\begin{align*}
|B_1(x)| & \lesssim \int_{q\in \Pi(6B_0)}  \frac{\ve^{1/4}
D(q)}{\bigl(D(q) +
|x-\wt A(q)|\bigr)^{n+1}}\,d\sigma(q) \\
& \lesssim\int_{z\in \wt F\cap 7B_0} \frac{\ve^{1/4} \ell(z)}{\bigl(\ell(z) +
|x-z|\bigr)^{n+1}}\,d\mu(z).
\end{align*}
Consider now the operator
\begin{equation}\label{eqops1}
Sf(x) = \int_{z\in \wt F \cap 7B_0} \frac{\ve^{1/4} \ell(z)}{\bigl(\ell(z) +
|x-z|\bigr)^{n+1}}\,f(z)\,d\mu(z),\qquad x\in\Gamma\cap 4B_0.
\end{equation}
 It is easy to check that its adjoint
satisfies
$$|S^* f(z)| \lesssim \ve^{1/4}Mf(z),$$
where $M$ stands for the Hardy-Littlewood maximal operator
$$Mf(z) =\sup_{r>0}\frac1{r^n} \int_{B(z,r)\cap\Gamma\cap
4B_0}|f|\,d\HH^n,$$ which is bounded from $L^2(\Gamma\cap 4B_0)$ into
$L^2(\mu_{|\wt F\cap 7B_0})$, and so $$\|S^*\|_{L^2(\Gamma\cap
4B_0),L^2(\mu_{|\wt F\cap 7B_0})}\lesssim\ve^{1/4}.$$ Therefore,
$S:L^2(\mu_{|\wt F\cap 7B_0})\to L^2(\Gamma\cap 4B_0)$ is bounded
with norm $\lesssim\ve^{1/4}$ and then
$$
\int_{x\in\Gamma\cap 4B_0}\biggl( \int_{y\in \wt F\cap 7B_0}
\frac{\ve^{1/4}d(y)}{d(y) + |y-x|^{n+1}}\,d\mu(y)\biggr)^2 d\HH^n(x)
 \lesssim \ve^{1/2} r_0^n.
$$
Thus,
$$\|B_1\|_{L^2(\Gamma\cap 4B_0)}\lesssim \ve^{1/4} r_0^{n/2}.$$

We deal now with the second term on the right hand side of \rf{eqaux88}, with $f(y) = K^\bot_{\ell(x),r_0} (x-y)$:
$$B_2(x):=
\biggl|
\int_{x\in\Pi(6B_0)} K^\bot_{\ell(x),r_0} (x-\wt A(p))\,b(p) \,dp\biggr|.$$
By the $L^2$-boundedness of Riesz transforms on $L^2(\Gamma)$ and the fact that $\|b\|_\infty\lesssim\ve^{1/4}$, we get
$$\|B_2\|_{L^2(\Gamma\cap 4B_0)}\lesssim \|\chi_{\Pi(6B_0)}b\|_2 \leq
\ve^{1/4} r_0^{n/2}.$$

Finally we deal with the third term on the right side of \rf{eqaux88}:
$$B_3(x) =
\int_{5B_0\cap \wt F} \bigl|K^\bot_{\ell(x),r_0} (x-P(y)) - K^\bot_{\ell(x),r_0} (x-y)\bigr|\,d\mu(x).$$
Since
$$|\nabla K^\bot_{\ell(x),r_0}(z)| \lesssim
\frac1{\bigl(\ell(x) + |z|\bigr)^{n+1}}$$ for all $z\in\R^d$, and
$|y-P(y)|\leq C\dist(y,\Gamma)\leq C\ve^{1/2}d(y)$ for $y\in\wt F$, we deduce that
\begin{align*}
\bigl|K^\bot_{\ell(x),r_0} (x-y) - K^\bot_{\ell(x),r_0}
(x-P(y))\bigr|  & \lesssim \frac{|y-P(y)|}{\bigl(\ell(x) + |y-x|\bigr)^{n+1}}\\
& \lesssim \frac{\ve^{1/2}d(y)}{\bigl(\ell(x) + |y-x|\bigr)^{n+1}}.
\end{align*}
Therefore,
\begin{multline*}
|R^\bot_{\ell(x),r_0} (\mu_{|\wt F\cap 5B_0})(x) -
R^\bot_{\ell(x),r_0} (P_\#\mu_{|\wt F\cap 5B_0})(x)| \\ \lesssim
\int_{y\in \wt F\cap 5B_0} \frac{\ve^{1/2}d(y)}{\bigl(\ell(x) +
|y-x|\bigr)^{n+1}}\,d\mu(y).
\end{multline*}
Recall that $\ell(x)=10D(\Pi(x))$, and since $D(\cdot)$ is
$1$-Lipschitz,
$$D(\Pi(y))\lesssim D(\Pi(x))+|x-y| = 10\ell(x) + |x-y|.$$
For $y\in F$, by Lemma \ref{lemtec} we infer that $d(y)\approx
D(\Pi(y))$, and so
$$d(y)\lesssim \ell(x) + |x-y|.$$
Thus,
\begin{align*}
\|B_3\|_{L^2(\Gamma\cap 4B_0)}^2 \lesssim \int_{x\in\Gamma\cap 4B_0}\biggl( \int_{y\in \wt F\cap
5B_0} \frac{\ve^{1/2}d(y)}{\bigl(d(y) + |y-x|\bigr)^{n+1}}\,d\mu(y)\biggr)^2
d\HH^n(x).
\end{align*}
Let $T:L^2(\mu_{|\wt F\cap 5B_0})\to L^2(\Gamma\cap 4B_0)$ be the
following operator
\begin{equation}\label{eqdeftt}
Tf(x) = \int_{y\in \wt F\cap 5B_0} \frac{\ve^{1/2}d(y)}{\bigl(d(y) +
|y-x|\bigr)^{n+1}}\,d\mu(y).
\end{equation}
Arguing as in the case of the operator $S$ from \rf{eqops1}, it is easy to check that
$T:L^2(\mu_{|\wt F\cap 5B_0})\to L^2(\Gamma\cap 4B_0)$ is bounded
with norm $\lesssim\ve^{1/2}$ and then
 \begin{align*}
\int_{x\in\Gamma\cap 4B_0}\biggl( \int_{y\in \wt F\cap 5B_0}
\frac{\ve^{1/2}d(y)}{\bigl(d(y) + |y-x|\bigr)^{n+1}}\,d\mu(y)\biggr)^2 d\HH^n(x)
& \lesssim \ve \HH^n(\Gamma\cap 4B_0)\\ & \lesssim \ve r_0^n.
\end{align*}
Thus we obtain
$$\|B_3\|_{L^2(\Gamma\cap 4B_0)} \lesssim \ve^{1/2} r_0^{n/2}.$$
If we add the estimates obtained for $B_1$, $B_2$ and $B_3$, the lemma follows.
\end{proof}

\begin{lemma} \label{lemdhh}
We have
 \begin{equation} \label{eqt570}
 \|R^\bot_{\ell(\cdot),r_0}
(h\,d\HH^n_{|\Gamma\cap 5B_0}) - R^\bot_{\ell(\cdot),r_0}
\HH^n_{|\Gamma\cap 5B_0}\|_{L^2(\Gamma\cap 4B_0)} \lesssim \alpha^2
\|\nabla A\|_2 + \alpha^2 r_0^{n/2}.
\end{equation}
\end{lemma}

Let us remark that, for the arguments in Lemma \ref{lemf3petit} below,
it is important that the last term on the right side of \rf{eqt570}
is $\alpha^2 r_0^{n/2}$ instead of $\alpha r_0^{n/2}$, say.

\begin{proof}
By Lemma \ref{lemnormpet} we have
$$\|R^\bot_{\ell(\cdot),r_0}
\|_{L^2(\Gamma\cap 4B_0),L^2(\Gamma\cap 4B_0)}\lesssim \|\nabla
A\|_\infty.$$
 Thus
$$\|R^\bot_{\ell(\cdot),r_0}
(h\,d\HH^n_{|\Gamma\cap 5B_0}) - R^\bot_{\ell(\cdot),r_0}
\HH^n_{|\Gamma\cap 5B_0}\|_{L^2(\Gamma\cap 4B_0)} \lesssim \|\nabla
A\|_\infty \|h-1\|_{L^2(\Gamma\cap 6B_0)}.$$
 On the other hand, writing $p=\Pi(x)$ we have
$$ |h(x)-1|  = \biggl|\frac{g(p)}{J(\wt A)(p)} -1\biggr| \leq
\biggl|\frac{g(p)}{J(\wt A)(p)} -g(p)\biggr| + |g(p)-1|.$$ Recalling that
$\|J(\wt A) -1 \|_2\leq \|\nabla A\|_\infty \|\nabla A\|_2$ and
$\|\chi_{\Pi(8B_0)}(g-1)\|_2\lesssim \alpha r_0^{n/2}$, we get
$$\|h-1\|_{L^2(\Gamma\cap 6B_0)} \lesssim \alpha \|\nabla
A\|_2 + \alpha r_0^{n/2},$$ and thus
$$ \|R^\bot_{\ell(\cdot),r_0}
(h\,d\HH^n_{|\Gamma\cap 5B_0}) - R^\bot_{\ell(\cdot),r_0}
\HH^n_{|\Gamma\cap 5B_0}\|_{L^2(\Gamma\cap 4B_0)} \lesssim \alpha^2
\|\nabla A\|_2 + \alpha^2 r_0^{n/2}.$$
\end{proof}

% ********************************************************************
% ********************************************************************

\subsection{The implication
$\|R^\bot_{\ell(\cdot),r_0} \mu_{\wt F}\|_{L^2(\Gamma)}$ big
$\Rightarrow$ $\|R^\bot_{\ell(\cdot),r_0} \mu\|_{L^2(\mu_{|F})}$
big} \label{sub115}

Recall that on $4B_0$, the image measure of $P_\#\mu_{|\wt F}$ by $\Pi$
coincides with $\sigma$ and that $h\,d\HH^n_{|\Gamma\cap 5B_0} =
P_\# \bigl(g(x)\,dx\bigr)$, with $g(x)=\bigl(\vphi_{\ve^{1/4} D(x)}
* \sigma\bigr)(x)$.
We denote
$$G_1= \{p\in \Pi(8 B_0):g(p)>1/2\},$$
and
$$G_0=P(G_1).$$

\begin{lemma}\label{lemmidag}
We have
$$\HH^n(\Gamma\cap 6B_0\setminus G_0) \lesssim \alpha^2\,r_0^n.$$
\end{lemma}

\begin{proof}
By \rf{eqgnorm} we have
$$\int_{\Pi(8 B_0)}|g-1|\,dx\leq C\alpha ^2\,r_0^n.$$
Thus,
\begin{align*}
\LL^n(\Pi(8 B_0)\setminus G_1) & \leq \LL^n\{p\in \Pi(8 B_0): |g(p)-1|>1/2\} \\
& \leq 2\int_{\Pi(8B_0)}|g-1|\,dp\leq C\alpha ^2\,r_0^n.
\end{align*}
It is clear that then we also have
$$\HH^n(\Gamma\cap 6B_0\setminus P(G_1))\lesssim \alpha^2\,r_0^n.$$
\end{proof}

\begin{lemma}\label{lemaprox2}
$$\|R^\bot_{\ell(\cdot),r_0} (\mu_{|\wt F \cap 5B_0})\|_{L^2(\Gamma\cap 4B_0\cap
G_0)}\lesssim \|R^\bot_{\ell(\cdot),r_0} (\mu_{|\wt F\cap
5B_0})\|_{L^2(\mu_{|\wt F\cap 4B_0})} + \ve^{1/8} r_0^{n/2}.$$
\end{lemma}

\begin{proof}
We denote $f(x) = R^\bot_{\ell(\cdot),r_0} (\mu_{|5B_0\cap \wt F})(x)$.
Since $h(x)>1/3$ on $G_0$, we have
 \begin{align*}
 \|f\|_{L^2(\Gamma\cap 4B_0\cap G_0)}^2 & \leq 3\int_{\Gamma\cap 4B_0}
 |f|^2\,h\,d\HH^n \\
 & \leq 3\biggl|\int_{\Gamma\cap 4B_0}
 |f|^2\,h\,d\HH^n - \int_{\wt F\cap 4B_0}
 |f|^2\,d\mu\biggr| +
3\int_{\wt F\cap 4B_0}
 |f|^2\,d\mu.
 \end{align*}
To prove the lemma it is enough to show that
\begin{equation}\label{eqeno24}
I:= \biggl|\int_{\Gamma\cap 6B_0}
 |f|^2\,h\,d\HH^n - \int_{\wt F\cap 6B_0}
 |f|^2\,d\mu\biggr| \lesssim \ve^{1/8} r_0^{n/2}.
\end{equation}
To this end we will use Lemma \ref{lemauxxx}, with $|f|^2$ instead
of $f$, and with $6B_0$ replacing $5B_0$, and $7B_0$ replacing
$6B_0$. Notice that $\supp(f)\subset 6B_0$. It is clear that Lemma
\ref{lemauxxx} also holds in this situation. So we have
\begin{align*}
I & \leq C
\iint_{\begin{subarray}{l}
p\in\Pi(7B_0) \\|p-q|\leq \ve^{1/4} D(q)\end{subarray}}\frac{1}{\bigl(\ve^{1/4} D(q)\bigr)^n}\, \bigl||f(\wt A(p))|^2 - |f(\wt A(q))|^2\bigr|\,d\sigma(q)dp\\
&\quad + \biggl|
\int_{p\in\Pi(7B_0)} |f(\wt A(p))|^2\,b(p) \,dp\biggr| +\int_{7B_0\cap \wt F} \bigl||f(P(x))|^2 - |f(x)|^2\bigr|\,d\mu(x)
\\ & =: C\,I_1+I_2+I_3,
\end{align*}
where $b(p)$ is some function satisfying $\|b\|_\infty\lesssim \ve^{1/4}$.

First we estimate $I_1$. Setting
\begin{equation}\label{eqwx59}
\bigl||f(\wt A(p))|^2 - |f(\wt A(q))|^2\bigr| \leq
\bigl|f(\wt A(p)) - f(\wt A(q))\bigr|\times
\bigl(|f(\wt A(p))| + |f(\wt A(q))|\bigr)
\end{equation}
and applying Cauchy-Schwartz, we get
\begin{align}\label{eqwx29}
I_1& \leq \biggl(
\iint_{\begin{subarray}{l}
p\in\Pi(7B_0) \\|p-q|\leq \ve^{1/4} D(q)\end{subarray}}\frac1{\bigl(\ve^{1/4} D(q)\bigr)^n}\, \bigl|f(\wt A(p)) - f(\wt A(q))\bigr|^2\,d\sigma(q)dp\biggr)^{1/2} \\
&\quad \times
\biggl(
\iint_{\begin{subarray}{l}
p\in\Pi(7B_0) \\|p-q|\leq \ve^{1/4} D(q)\end{subarray}}\frac1{\bigl(\ve^{1/4} D(q)\bigr)^n}\, \bigl(|f(\wt A(p))|+ |f(\wt A(q))|\bigr)^2\,d\sigma(q)dp\biggr)^{1/2}\nonumber\\
& =: I_{1,1}^{1/2} \times I_{1,2}^{1/2}.\nonumber
\end{align}
To estimate $I_{1,1}$ notice that if
$|p-q|\leq \ve^{1/4} D(q)$, then
\begin{equation}\label{eqwx56}
\bigl|f(\wt A(p))\! -\! f(\wt A(q))\bigr|  = \bigl|R^\bot_{\ell(\cdot),r_0} \mu_{|5B_0\cap \wt F}(\wt A(p))-
R^\bot_{\ell(\cdot),r_0} \mu_{|5B_0\cap \wt F}(\wt A(q))\bigr| \! \lesssim  \ve^{1/4}.
\end{equation}
For this inequality notice $D(p) \approx D(q)$ because $|p-q|\leq
\ve^{1/4} D(q)$, and recall also that $\ell(x) = 10D(\Pi(x))$. We
leave the details for the reader. Therefore,
\begin{multline*}
I_{1,1} = \iint_{\begin{subarray}{l}
p\in\Pi(7B_0) \\|p-q|\leq \ve^{1/4} D(q)\end{subarray}}\frac1{\bigl(\ve^{1/4} D(q)\bigr)^n}\, \bigl|f(\wt A(p)) - f(\wt A(q))\bigr|^2\,d\sigma(q)dp  \\ \lesssim \ve^{1/2}
\iint_{\begin{subarray}{l}
p\in\Pi(7B_0) \\|p-q|\leq \ve^{1/4} D(q)\end{subarray}}\frac1{\bigl(\ve^{1/4} D(q)\bigr)^n}\,d\sigma(q)dp
\lesssim \ve^{1/2}\,r_0^n.
\end{multline*}
To deal with $I_{1,2}$ we set
\begin{align}\label{eqwx35}
I_{1,2}^{1/2} & \leq \biggl(
\iint_{\begin{subarray}{l}
p\in\Pi(7B_0) \\|p-q|\leq \ve^{1/4} D(q)\end{subarray}}\frac1{\bigl(\ve^{1/4} D(q)\bigr)^n} |R^\bot_{\ell(\cdot),r_0} \mu_{|5B_0\cap \wt F}(\wt A(p))|^2\,d\sigma(q)dp\biggr)^{1/2}\\
&\quad
+
\biggl(
\iint_{\begin{subarray}{l}
p\in\Pi(7B_0) \\|p-q|\leq \ve^{1/4} D(q)\end{subarray}}\frac1{\bigl(\ve^{1/4} D(q)\bigr)^n} |R^\bot_{\ell(\cdot),r_0} \mu_{|5B_0\cap \wt F}(\wt A(q))|^2\,d\sigma(q)dp\biggr)^{1/2}\nonumber\\
& =: I_{1,2,a}^{1/2} + I_{1,2,b}^{1/2}.\nonumber
\end{align}
Concerning $I_{1,2,a}$, we have
$$I_{1,2,a}^{1/2} \lesssim  \biggl(
\int_{p\in\Pi(7B_0)} |R^\bot_{\ell(\cdot),r_0} \mu_{|5B_0\cap \wt
F}(\wt A (p))|^2\,dp\biggr)^{1/2}\lesssim r_0^{n/2},$$ by the $L^2$
boundedness of Riesz transforms from $L^2(\mu_{|\wt F})$ into
$L^2(\Gamma)$. For the last integral in the right side of
\rf{eqwx35} we take into account that $D(p)\approx D(q)$, and then
we get
\begin{align}\label{eqwx98}
& I_{1,2,b}^{1/2}  \lesssim \biggl(
\int_{q\in\Pi(7.5B_0)} |R^\bot_{\ell(\cdot),r_0} \mu_{|5B_0\cap \wt F}(\wt A (q))|^2\,d\sigma(q)\biggr)^{1/2} \\
&  = C \biggl(\int_{q\in\Pi(7.5B_0)} |f(\wt A(q))|^2\,d\sigma(q) \biggr)^{1/2} \leq C \biggl(\int_{y\in 8B_0} |f(P(y))|^2\,d\mu_{\wt F}(y)\biggr)^{1/2} \nonumber\\
& \leq C \biggl(\int_{y\in 8B_0} |f(P(y)) - f(y)|^2\,d\mu_{\wt
F}(y)\biggr)^{1/2} + C \biggl(\int_{y\in 8B_0} |f(y)|^2\,d\mu_{\wt
F}(y)\biggr)^{1/2}.\nonumber
\end{align}
Using the $L^2(\mu_{|\wt F})$ boundedness of Riesz transforms, the last integral is $\lesssim r_ 0^n$.
For the first one we argue as in \rf{eqwx56}: given $y\in\wt F$, we have $|y-P(y)|\lesssim \ve^{1/2}\,d(y)
\approx \ve^{1/2}\ell(y)$, and then it easily follows that
$$\bigl|f(P(y)) - f(y)\bigr| = \bigl|R^\bot_{\ell(\cdot),r_0} \mu_{|5B_0\cap \wt F}(P(y))-
R^\bot_{\ell(\cdot),r_0} \mu_{|5B_0\cap \wt F}(y)\bigr| \! \lesssim  \ve^{1/2}.
$$
Therefore, the first term on the right side of \rf{eqwx98} is bounded above by $\ve^{1/2}r_0^{n/2}$, and so
$I_{1,2,b}^{1/2} \lesssim r_0^{n/2}$, and thus $I_{1,2}^{1/2} \lesssim r_0^{n/2}$. Recalling that
$I_{1,1}\leq \ve^{1/2}r_0^n$, we deduce that
$$I_1\lesssim \ve^{1/4} r_0^n.$$

To estimate the integral $I_2$ we use the fact that $\|b\|_\infty\lesssim\ve^{1/4}$, and so
$$I_2\lesssim \ve^{1/4}\int_{x\in\Pi(7B_0)} |f(\wt A(p))|^2 \,dp.$$
The last integral is similar to $I_{1,2,a}$, and thus we have
$$I_2\lesssim \ve^{1/4}\,r_0^n.$$

To deal with $I_3$ we argue as in \rf{eqwx59} and, similarly to \rf{eqwx29}, we infer that
\begin{align}\label{eqwx99}
I_3& \leq \biggl(
\int_{
x\in7B_0\cap \wt F} \bigl|f(P(x)) - f(x)\bigr|^2\,d\mu(x)\biggr)^{1/2} \\
&\quad \times
\biggl(
\int_{
x\in7B_0\cap \wt F} \bigl||f(P(x))|+ |f(x)|\bigr|^2\,d\mu(x)\biggr)^{1/2} =: I_{3,1}^{1/2} \times I_{3,2}^{1/2}.\nonumber
\end{align}
The integral $I_{3,1}$ is similar to the first one on right side of \rf{eqwx98}, and so we have
$I_{3,1}\lesssim \ve^{1/4}r_0^n$. For $I_{3,2}$ we set
$$I_{3,2} \leq \biggl(
\int_{
x\in7B_0\cap \wt F} \bigl|f(x)\bigr|^2\,d\mu(x)\biggr)^{1/2}
+
\biggl(
\int_{
x\in7B_0\cap \wt F} \bigl|f(P(x))\bigr|^2\,d\mu(x)\biggr)^{1/2}.$$
The first term on the right side is bounded above by $C\,r_0^{n/2}$, by the $L^2(\mu_{|\wt F})$ boundedness of
Riesz transforms, and for the second one we write
\begin{multline*}S:= \biggl(
\int_{
x\in7B_0\cap \wt F} \bigl|f(P(x))\bigr|^2\,d\mu(x)\biggr)^{1/2} \\
\leq
\int_{
x\in7B_0\cap \wt F} \bigl|f(x)\bigr|^2\,d\mu(x)\biggr)^{1/2}
+
\int_{
x\in7B_0\cap \wt F} \bigl|f(P(x)) - f(x)\bigr|^2\,d\mu(x)\biggr)^{1/2}.
\end{multline*}
As above, the first term satisfies $\lesssim r_0 ^{n/2}$, and the
second one coincides with $I_{3,1}^{1/2}$, and so we have $S\lesssim
r_0^{n/2}$. Thus, $I_{3,2}\lesssim r_0 ^{n/2}$, and then
$I_3\lesssim \ve^{1/8}r_0^n$.

If we gather the estimates obtained for $I_1$, $I_2$ and $I_3$, \rf{eqeno24} follows and we are done.
\end{proof}

% ********************************************************************
% ********************************************************************

\subsection{The proof that $F_3$ is small}

\begin{lemma}\label{lemf3petit}
We have
$$\mu(F_3)\leq\alpha^{1/2}\mu(F).$$
\end{lemma}

\begin{proof}
We will use all the results obtained Subsections \ref{sub111}-\ref{sub115}.
From \rf{eqfac2} and Lemma \ref{lemfff}, we deduce
\begin{equation}\label{eqwxx1}
\|\nabla A\|_2 \lesssim
\|R^\bot_{\ell(\cdot),r_0} (\HH^n_{\Gamma\cap
5B_0})\|_{L^2(\Gamma\cap 4B_0)} + C\alpha^2\, r_0^{n/2}.
\end{equation}
By Lemma \ref{lemmidag} and  since $R^\bot_{\ell(\cdot),r_0}$ is
bounded in $L^4(\Gamma)$ with norm $\lesssim\|\nabla
A\|_\infty\lesssim\alpha$, we deduce
\begin{align*}
 \|R^\bot_{\ell(\cdot),r_0} (\HH^n_{\Gamma\cap
5B_0}) & \|_{L^2(\Gamma\cap 4B_0\setminus G_0)}^2 \\
& \leq \|R^\bot_{\ell(\cdot),r_0} (\HH^n_{\Gamma\cap
5B_0})\|_{L^4(\Gamma\cap 4B_0)}^2 \,\HH^n(\Gamma\cap 4B_0\setminus
G_0)^{1/2}\\
& \lesssim \alpha^2 \HH^n(\Gamma\cap 5B_0)^{1/2}\,\HH^n(\Gamma\cap
4B_0\setminus G_0)^{1/2} \lesssim\alpha^{3}\, r_0^{n}.
\end{align*}
From this inequality and \rf{eqwxx1} we derive
$$\|\nabla A\|_2 \lesssim
\|R^\bot_{\ell(\cdot),r_0} (\HH^n_{\Gamma\cap
5B_0})\|_{L^2(\Gamma\cap 4B_0\cap G_0)} + \alpha^{3/2}\,
\mu(F)^{1/2},$$ since $\alpha\ll1$. This estimate and Lemmas
 \ref{lemaprox1} and \ref{lemdhh} imply that
$$\|\nabla A\|_2 \lesssim
\|R^\bot_{\ell(\cdot),r_0} \mu_{|\wt F\cap 5B_0}\|_{L^2(\Gamma\cap
4B_0\cap G_0)} +\alpha^2\|\nabla A\|_2 +
(\alpha^{3/2}+\alpha^2+\ve^{1/4})\mu(F)^{1/2}.$$ Thus, if $\alpha$
is small enough and $\ve^{1/4}\leq\alpha^{3/2}$, we get
$$\|\nabla A\|_2 \lesssim
\|R^\bot_{\ell(\cdot),r_0} \mu_{|\wt F\cap 5B_0}\|_{L^2(\Gamma\cap
4B_0\cap G_0)} + \alpha^{3/2}\,\mu(F)^{1/2}.$$ Together with Lemma
\ref{lemf3gran1} this implies that
\begin{equation}\label{eqwx53}
\mu(F_3) \lesssim \alpha^{-2}\|\nabla A\|_2^2 +
\ve^{1/2}\mu(F)\lesssim \alpha^{-2} \|R^\bot_{\ell(\cdot),r_0}
\mu_{|\wt F\cap 5B_0}\|_{L^2(\Gamma\cap 4B_0\cap G_0)}^2
+\alpha\,\mu(F).
\end{equation}
Recall that by Lemma \ref{lemaprox2}, we have
$$\|R^\bot_{\ell(\cdot),r_0} (\mu_{|5B_0\cap \wt F})\|_{L^2(\Gamma\cap 4B_0\cap
G_0)}\lesssim \|R^\bot_{\ell(\cdot),r_0} (\mu_{|\wt F\cap
5B_0})\|_{L^2(\mu_{|\wt F\cap 4B_0})} + \ve^{1/8} r_0^{n/2}.$$
From this estimate and \rf{eqwx53} we deduce that
\begin{equation}\label{eqqfii}
\mu(F_3)  \lesssim \alpha^{-2} \|R^\bot_{\ell(\cdot),r_0} (\mu_{|\wt
F\cap 5B_0})\|_{L^2(\mu_{|\wt F\cap 4B_0})}^2 +\alpha\,\mu(F)
\end{equation}
(assuming always $\ve\ll\alpha\ll1$).

\vvv
Now we denote
$$B_1=\{x\in \wt F: R_* \mu_{|5B_0\setminus \wt F}(x)>\ve^{1/4}\}.$$
By the boundedness of Riesz transforms from $M(\R^d)$ (the space of finite Borel measures on $\R^d$) into
$L^{1,\infty}(\mu_{|\wt F})$, we get
$$\mu(B_1)\lesssim \frac{\mu(5B_0\setminus \wt F)}{\ve^{1/4}} \lesssim \frac{\ve^{1/2}\mu(\wt F)}{\ve^{1/4}} =
\ve^{1/4}\mu(\wt F).$$
By Cauchy-Schwartz and the $L^4(\mu_{|\wt F})$ boundedness of Riesz transforms we get
$$\|R^\bot_{\ell(\cdot),r_0} (\mu_{|\wt F\cap
5B_0})\|_{L^2(\mu_{|B_1})}^2 \leq \|R^\bot_{\ell(\cdot),r_0} (\mu_{|\wt F\cap
5B_0})\|_{L^4(\mu_{|B_1})}^2 \mu(B_1)^{1/2} \lesssim \ve^{1/8}\,\mu(F).
$$
On the other hand, from \rf{eqqfii} we infer that
$$
\|R^\bot_{\ell(\cdot),r_0} (\mu_{|\wt F\cap 5B_0})\|_{L^2(\mu_{|\wt
F\cap 4B_0})}^2 \geq C^{-1}\alpha^2\bigr[\mu(F_3) -
C\alpha\,\mu(F)\bigr].$$

{\bf Suppose that {\boldmath $\mu(F_3)>\alpha^{1/2}\mu(F)$}}. Then
$\mu(F_3) - C\alpha\,\mu(F)\gtrsim \alpha^{1/2}\mu(F)$, and by the
preceding estimates we get
$$\|R^\bot_{\ell(\cdot),r_0} (\mu_{|\wt F\cap
5B_0})\|_{L^2(\mu_{|\wt F\cap 4B_0})}^2 \geq C^{-1}\alpha^{5/2}
\mu(F) \geq \frac12 \|R^\bot_{\ell(\cdot),r_0} (\mu_{|\wt F\cap
5B_0})\|_{L^2(\mu_{|B_1})}^2,$$ because $\ve^{1/8}\ll\alpha^{5/2}$.
Therefore,
\begin{align*}
 \|R^\bot_{\ell(\cdot),r_0} &(\mu_{|\wt F\cap
5B_0})\|_{L^2(\mu_{|\wt F\cap 4B_0\setminus B_1})}^2 \\ & = \|R^\bot_{\ell(\cdot),r_0} (\mu_{|\wt F\cap
5B_0})\|_{L^2(\mu_{|\wt F\cap 4B_0})}^2 - \|R^\bot_{\ell(\cdot),r_0} (\mu_{|\wt F\cap
5B_0})\|_{L^2(\mu_{|B_1})}^2 \\
&\geq\frac12 \|R^\bot_{\ell(\cdot),r_0} (\mu_{|\wt F\cap
5B_0})\|_{L^2(\mu_{|\wt F\cap 4B_0})}^2 \gtrsim \alpha^{5/2} \mu(F).
\end{align*}
Since $ R_* \mu_{|5B_0\setminus \wt F}(x)\leq\ve^{1/4}$ on $\wt F\setminus B_1$, we have
$$\|R^\bot_{\ell(\cdot),r_0} (\mu_{|5B_0 \setminus \wt F})\|_{L^2(\mu_{|\wt F\cap 4B_0\setminus B_1})}^2\leq
\ve^{1/2}\mu(4B_0).$$
Thus we deduce that
$$\|R^\bot_{\ell(\cdot),r_0}(\mu_{|5B_0})\|_{L^2(\mu_{|\wt F\cap 4B_0\setminus B_1})}^2 \geq (C^{-1}\alpha^{5/2} -C\ve^{1/2})\mu(F)\gtrsim
\alpha^{5/2}\mu(F).$$ Since
$R^\bot_{\ell(\cdot),r_0}(\mu_{|5B_0})(x) =
R^\bot_{\ell(\cdot),r_0}\mu(x)$ for any $x\in4B_0$, we have
$$\|R^\bot_{\ell(\cdot),r_0}\mu(x)\|_{L^2(\mu_{|\wt F\cap 4B_0\setminus B_1})}^2 \gtrsim
\alpha^{5/2}\mu(F),$$ which contradicts the assumption (d) in Main
Lemma \ref{mlem}.
\end{proof}

% ********************************************************************
% ********************************************************************
% ********************************************************************

\end{document}